\documentclass[a4paper,11pt, reqno]{amsart}

\usepackage[ansinew]{inputenc}
\usepackage{amsmath,amsthm,amsfonts,amssymb,stackengine}
\usepackage{mathrsfs}
\usepackage[left=2.50cm,right=2.25cm,top=2.50cm,bottom=2.50cm]{geometry}
\usepackage{graphicx}
\usepackage[english]{babel}
\usepackage{hyperref}
\usepackage{enumitem,fancyhdr}
\usepackage{etoolbox,lineno}

\usepackage{fouridx}

\usepackage{soul}
\usepackage[dvipsnames]{xcolor}

\newcommand\red[1]{{\color{red}#1}}

\numberwithin{equation}{section}
\theoremstyle{plain}
\newtheorem{Thm}{Theorem}[section]

\newtheorem{Prop}[Thm]{Proposition}
\newtheorem{Lem}[Thm]{Lemma}
\newtheorem{Cc}[Thm]{Corollary}
\theoremstyle{definition}
\newtheorem{Def}[Thm]{Definition}
\newtheorem{Rem}[Thm]{Remark}
\newtheorem{Rk}[Thm]{Remark}

\newtheorem{assum}[Thm]{Assumption}

\newcommand{\rK}{\mathrm{K}}

\newcommand{\divv}{\mathrm{div }\;}

\newcommand{\w}{\mathbf{W}}

\newcommand{\bv}{\mathbf{v}}

\newcommand{\el}{\mathbf{L}}
\newcommand{\ve}{\mathbf{V}}
\newcommand{\h}{\mathbf{H}}

\newcommand{\mo}{\mathcal{O}}

\newcommand{\bu}{\mathbf{u}}

\newcommand{\eps}{\varepsilon}

\newcommand{\E}{\mathbb{E}}

\newcommand{\G}{\mathbf{G}_{\alpha}}

\newcommand{\rA}{\mathrm{A}}

\newcommand{\ydela}{\mathbf{y}^{\alpha, \delta}}
\newcommand{\zdela}{\mathbf{z}^{\alpha, \delta}}

\newcommand{\valph}{J_\alpha^{-1} \bu }
 \newcommand{\bp}{\mathbf{p}}
 \newcommand{\CO}{{{ \mathcal O }}}

\newcommand{\BF}{{{ \mathbb{F} }}}
\newcommand{\BP}{{{ \mathbb{P} }}}
 \newcommand{\lqq}{\lefteqn}

\allowdisplaybreaks
\allowdisplaybreaks
\def\XXint#1#2#3{{\setbox0=\hbox{$#1{#2#3}{\int}$ }
		\vcenter{\hbox{$#2#3$ }}\kern-.6\wd0}}

\usepackage{parskip}

\begin{document}
	\title{Large, moderate deviations principle  and $\alpha$-limit  for the 2D Stochastic LANS-$\alpha$}
	\author{Z.I. Ali}
	 \address{Z.I. Ali, Department of Mathematical Sciences,
		University of South Africa,  Florida 0003, South Africa.}
	\email{alizi@unisa.ac.za}
	\author{P.A. Razafimandimby}
	\address{P. Razafimandimby, School of Mathematical Sciences,
		Dublin City University.}
	\email{paul.razafimandimby@dcu.ie}
	\author{T.A. Tegegn}
	\address{T.A. Tegegn, Department of Mathematics and Applied Mathematics,
		Sefako Makgatho Health Sciences University, South Africa.}
	\email{tesfalem.tegegn@smu.ac.za}
	\maketitle
	
\begin{abstract}
In this paper we consider the  Lagrangian Averaged Navier-Stokes Equations, also known as, LANS-$\alpha$ Navier-Stokes model on the two dimensional torus. We assume that  the  noise is a cylindrical Wiener process and its coefficient is multiplied by $\sqrt{\alpha}$. We then study through the lenses of the large and moderate deviations principle the behaviour of the trajectories of the solutions of the stochastic system as $\alpha$ goes to 0. Instead of giving two separate proofs of the two deviations principles we  present a unifying approach to the proof of the LDP and MDP  and express the rate function in term of the unique solution of the Navier-Stokes equations. Our proof is based on the weak convergence approach to large deviations principle. As a by-product of our analysis we also prove that the solutions of the stochastic LANS-$\alpha$ model converge in probability to the solutions of the deterministic Navier-Stokes equations.
\end{abstract}

	
	
	\vspace*{1cm}
	
	\hrule
	
	\vspace*{1cm}
	
	
	\pagestyle{fancy}
	\fancyhead{}\fancyfoot{}

	\lhead{Z. Ali , P. Razafimandimby \& T. Tegegn}
	\rhead{LDP results for the stochastic LANS-$\alpha$}
	\cfoot{\arabic{page}}
\section{Introduction}
The Navier-Stokes system is  the most used model in turbulence theory. 
In recent years, regularization models of Navier-Stokes equations such as the Navier-Stokes-$\alpha$, Leray-$\alpha$, modified Leray-$\alpha$, Clark-$\alpha$ to name a few,   were introduced as subgrid models scale of the Navier-Stokes equations (NSE), see, for instance, \cite{CLT}, \cite{CAO-TITI}, \cite{Chen-1}, \cite{Chen-2}, \cite{Chen-3}, \cite{FOIAS-HOLM-TITI}. Numerical analyses in 
\cite{CHMZ, GH, HT, LL-0, LKT, LKTT, MKSM, NS} seem to confirm that the previous examples of $\alpha$-models  can capture remarkably well the physical phenomenon of turbulence in fluid flows at a lower
computational cost. 

In this paper we will consider a stochastic version of  the Navier-Stokes-$\alpha$ (also known as the LANS-$\alpha$). 
More precisely we let $\CO=[0,2\pi]^2$ be the $2D$ torus, we fix an arbitrary time horizon $T\in(0,\infty)$ and we consider the following system
\begin{align} 
	\left\{ \begin{array}{ll}
		d\bv^\alpha + [-\Delta\bv^\alpha+\bu^\alpha\cdot\nabla\bv^\alpha + \sum_{j=1}^2\bv^\alpha_j\nabla\bu^\alpha+\nabla\bp^\alpha]dt = \alpha^\frac12G(\bu^\alpha)dW&\\
		\bv^\alpha = \bu^\alpha-\alpha^2\Delta\bu^\alpha&\\
		\divv \bu^\alpha = 0&\\
		\int_\CO\bu^\alpha(x)dx = 0&\\
		\bu^\alpha(t=0) = \xi,
	\end{array}\right.\label{Eq: LANS-ALPHA}
\end{align}
where $\bu^\alpha,\,\bp^\alpha$ are the fluid velocity and fluid pressure, respectively.
The symbol $W$ represents the cylindrical Wiener process evolving on a given separable Hilbert space $\mathrm{K}$.
The noise coefficient is a nonlinear map defined and taking values on Hilbert spaces that will be given later.
The symbol $\alpha$ is a small positive  parameter. 

 Observe that  when $\alpha=0$ the above system reduces to the deterministic 2D Navier-Stokes equations (NSEs):
 \begin{align} 
 \left\{ \begin{array}{ll}
 d\bu + [-\Delta\bu+\bu\cdot\nabla\bu + \nabla\bp ]dt = 0&\\
 \divv \bu = 0&\\
 \int_\CO\bu(x)dx = 0&\\
 \bu(t=0) = \xi,
 \end{array}\right.\label{Eq:NSE}
 \end{align}
  Thus, we expect that a sequence of solutions to the system \eqref{Eq: LANS-ALPHA} will converge in appropriate sense to a solution to \eqref{Eq:NSE} as $\alpha \to 0$. For the deterministic case, \textit{i.e.}, when $G\equiv 0$, it is known from \cite{FOIAS-HOLM-TITI} that then as $\alpha\to 0$ a weak solution to the deterministic  $3D$ LANS-$\alpha$ \eqref{Eq: LANS-ALPHA} model converges to $3D$ Navier-Stokes equations.
In \cite{CAO-TITI}, when $G\equiv 0$ the rate of convergence of the unique solution $2D$ \eqref{Eq: LANS-ALPHA} to the unique solution to the $2D$ Navier-Stokes equations was studied.
For the stochastic models, it was proved in \cite{CARABALLO1} that the stochastic $3D$ \eqref{Eq: LANS-ALPHA} has a unique strong solution when the noise coefficient $G$ is globally Lipschitz. 
When $G$ is only continuous, it was proved in \cite{Gabriel} that the stochastic $3D$ \eqref{Eq: LANS-ALPHA} has a global weak (or martingale) solutions. 
Furthermore, it is shown in \cite{GABU-SANGO} that when $\alpha\to 0$ a sequence of weak (or martingale) solutions of the stochastic $3D$ \eqref{Eq: LANS-ALPHA} model converges in distribution to a weak solution (or martingale) of the $3D$ Navier-Stokes equations. In the above references, the coefficient of the noise is not allowed to converges to 0 as $\alpha \to 0$. 

Our main goal in this paper is to study the behaviour of the solutions $\bu^\alpha$ to the system \eqref{Eq: LANS-ALPHA} as $\alpha \to 0$  through the lenses of the Large and Moderate Deviations Principle (LDP and MDP).   For this purpose we assume that the coefficient of the noise is multiplied by the square root of $\alpha$, \textit{i.e.}, of the form  $\alpha^\frac12G(\bu^\alpha)$. We then  analyse the asymptotic behaviour, as $\alpha \to 0$, of  the trajectories family of  $(\bu^\alpha)_{\alpha \in (0,1]}$ and $\left(\alpha^{-\frac12}\lambda^{-1}(\alpha)[\bu^\alpha-\bu]\right)_{\alpha \in (0,1]}$ where  $\lambda: (0,1] \to (0,\infty) $ is  a function satisfying 
	\begin{equation}\label{Eq:Def-Lambda}
	\lambda(\alpha) \rightarrow \infty \text{ and } \alpha^\frac12 \lambda(\alpha) \rightarrow 0 \text{ as } \alpha \rightarrow 0,
	\end{equation} 
	and $\bu$ is the solution to the deterministic NSE with initial data $\xi$. Thus, our goal and results in the present paper are different from the results in \cite{GABU-SANGO} and from results from several papers dealing with the deviation principles of $\alpha$-models of Navier-Stokes equations, see for instance \cite{Chueshov+Millet} and \cite{YANG-ZHAI}. 
	
	Roughly speaking, in the study of the MDP one is interested in probabilities of deviations of lower speed than in the classical LDP. In small diffusion (the coefficient of the noise is usually multiplied by $\alpha^\frac12$) the speed for the LDP is usually of order $\alpha$. The speed for the MDP is of order $\lambda^2(\alpha)$ and is provided by an LDP result for  $\left(\alpha^{-\frac12}\lambda^{-1}(\alpha)[\bu^\alpha-\bu]\right)_{\alpha \in (0,1]}$. 
Observe that since $\lambda(\alpha)$ converges to $\infty$ as slow as desired, then the MDP bridges the gap between the Central Limit Theorem and the LDP.	We refer, for instance, to \cite{Guillin} and \cite{Klebaner} for more detailed explanation and historical account of the MDP. We refer, for instance, to  \cite{Bessaih+Millet-2009}, \cite{Bessaih+Millet-2012}, \cite{ZB+BG+TJ-2017}, \cite{BuDu}, \cite{AB+PD+VM},  \cite{Chueshov+Millet},  \cite{Duan+Millet},  \cite{Lee+Leila}, \cite{Liu+Roeckner},  \cite{TZhang-NSE-MDP}, \cite{Wang-RDE_MDP}, \cite{Yang+Xueke}, \cite{TZhang-1}, \cite{Zhai-MDP} and references therein for a small sample of results from the extensive literature devoted to MDP and LDP for stochastic differential equations with small noise.

In several papers about LDP and MDP for stochastic system, the authors  usually present two separate proofs of the two deviations principles. In this present paper, instead of presenting two separate proofs of the LDP and MDP results we present a unifying approach for these deviation principles for the LANS-$\alpha$ model. A similar approach was introduced in \cite{PAUL} for the vanishing viscosity limit of the  second grade fluid. 
To be precise, we fix  $\delta\in \{0,1\}$ and consider the following problem   
\begin{align}
\left\{ 
\begin{array}{ll}
d\ydela + \left[\rA\ydela + \lambda_\delta(\alpha)\tilde{B}_\alpha(\ydela,\zdela) + \delta\left[\tilde{B}_\alpha(\bu,\zdela) +\tilde{B}_\alpha(\ydela,\valph)\right] \right] &\\
\qquad = -\lambda_\delta^{-1}(\alpha)\delta[\tilde{B}_\alpha(\bu,\valph)-B(\bu,\bu)]dt 
 + \alpha^\frac12\lambda_\delta^{-1}(\alpha) \G (\delta\bu+\lambda_\delta(\alpha)\ydela)dW,&\\
 \zdela = \ydela + \alpha^2\rA\ydela,&\\
\ydela(t=0) = (1-\delta)\xi, &
\end{array}
\right. \label{Eq: ABS-SVE-0}
\end{align}
where 
\begin{itemize}
	\item $\bu$ is the unique solution to the deterministic NSE with initial data $\xi$; 
	\item $\rA$ is the Stokes operator,  $J_\alpha =( I+\alpha^2\rA)^{-1};$ 
	\item $B(u,v)$ is roughly speaking the projection of $u\cdot \nabla v$ into the space of divergence free functions; 
	\item $\widetilde{B}(u,v)$ is the projection of $u\cdot \nabla v + \sum_{j=1}^2 v_j \nabla u_j $ into the space of divergence free functions;
	\item finally, $\widetilde{B}_\alpha(u,v)=J_\alpha \widetilde{B}(u,v)$ and $\mathbf{G}_\alpha(u)=J_\alpha G(u)$.
\end{itemize}
The major part of the paper is devoted to the proof of LDP result for the system \eqref{Eq: ABS-SVE-0} which provides the LDP and MDP results for the LANS-$\alpha$ model \eqref{Eq: LANS-ALPHA}. 
In fact, we observe that:
\begin{itemize}
	\item when $\delta=0$, the unique solution to \eqref{Eq: ABS-SVE-0} is exactly the unique solution to the LANS-$\alpha$ \eqref{Eq: LANS-ALPHA}. Thus, the LDP results for system \eqref{Eq: LANS-ALPHA} follows from the LDP result for the system \eqref{Eq: ABS-SVE-0} when $\delta=0$. 
	\item When $\delta=1$, the unique solutions to \eqref{Eq: ABS-SVE-0} is exactly $\alpha^{-\frac12}\lambda^{-1}(\alpha)[\bu^\alpha-\bu]$ where $\bu^\alpha$ and $\bu$ are the unique solution to \eqref{Eq: LANS-ALPHA}  and the deterministic NSE with initial data $\xi$, respectively. Hence,  the MDP result for \eqref{Eq: LANS-ALPHA} follows from the LDP results for the system \eqref{Eq: ABS-SVE-0} when $\delta=1$. 
\end{itemize}
Our main result  is stated in Theorem  \ref{THM:LDP1} whose  proof is presented in Section \ref{Sec:Proof-of-Main-Thm} and based on weak convergence approach to LDP and Budhiraja-Dupuis' results on representation of functionals of Brownian motion, see \cite{BuDu} and \cite{AB+PD+VM}. Also, we closely follow the techniques presented in the recent paper \cite{ZB+BG+TJ-2017}. Note however that our results do not fall into the framework of these papers or the results in \cite{Bessaih+Millet-2009}, \cite{Bessaih+Millet-2012},  \cite{Chueshov+Millet}, \cite{PAUL} \cite{TZhang-NSE-MDP}, \cite{YANG-ZHAI}. 
The authors of the papers \cite{Chueshov+Millet}, \cite{TZhang-NSE-MDP} and \cite{YANG-ZHAI} study the LDP or MDP of the Navier-Stokes equations and other hydroddynamical models, but they physical parameters such as the viscosity in their equations are not allowed to vanish.  The papers  \cite{Bessaih+Millet-2009}, \cite{Bessaih+Millet-2012} and \cite{PAUL} treat the LDP and zero viscosity limit of the shell models, the Naver-Stokes equations and the second grade fluids, respectively.

It is also worth pointing out that even though we rely on the abstract results in \cite{BuDu} and \cite{AB+PD+VM} our analysis are not trivial. Our results require the derivation of  uniform estimates on the difference between the terms in the LANS-$\alpha$ model and the Navier-Stokes equations. Due to the unifying approach to the LDP and MDP we present in this paper, these crucial  estimates are not available from previous works. 
We also note that as a by-product of our analysis we also show that the solution to \eqref{Eq: LANS-ALPHA} converges in probability to the unique solution to the deterministic NSE with initial data $\xi$ as $\alpha \to 0$, see Lemma \ref{Lem-Prop-Imp-1} and Remark \ref{Rem-ConvProba}. Of course, since we are in the two dimensional case this result is stronger than what was proved in \cite{GABU-SANGO}.

To close this introduction we now outline the layout of the paper. We introduce the necessary notations and the basic model in section \ref{sec:2}. In the same section we also give several preliminary results which are crucial for the subsequent analysis.  In section \ref{Assum-Noise}, we introduce the standing assumption on teh noise and state and prove a theorem on the existence and uniqueness of the solution to the problem \ref{Eq: ABS-SVE-0}. Section \ref{sec:4 Analysis} is devoted to the study of auxiliary deterministic and stochastic controlled systems. The results in section \ref{sec:4 Analysis} are important not only for the  the description of the rate functions associated to the LDP and MDP results but also for their proofs. Section \ref{sec:5Deviation} contains the main results and their proofs. Therein, we show  the convergence in probability of the solution to the stochastic LANS-$\alpha$ to the unique solution of the deterministic Navier-Stokes equations. By using the weak convergence approach we also prove in Section \ref{sec:5Deviation} 
that the solution of \eqref{Eq: ABS-SVE-0} satisfies the LDP on $C([0,T]; \h)\cap L^2(0,T; \ve)$. This LDP result for \eqref{Eq: ABS-SVE-0} provides the LDP and MDP results for the problem \eqref{Eq: LANS-ALPHA}. 

\section{Notations, the basic problems and some key estimates} \label{sec:2}

\subsection{Notation and the basic problems}
{We introduce necessary definitions of functional spaces frequently used in
	this work.  
	
	For a topological vector space $X$ we denote by $X'$ its dual space and we denote by $\langle u, u*\rangle_{X'}$ the duality paring between $u \in X$ and $u*\in X'$. 
	
	Throughout this paper we denote by $L^p(\mo; \mathbb{R}^2)$ and ${W}^{m,p}(\mo; \mathbb{R}^2)$, $p\in [1,\infty]$, $m\in \mathbb{N}$, the Lebesgue and Sobolev spaces of functions defined on $\mo$ and taking values in $\mathbb{R}^2$. The spaces of $u\in L^p(\mo; \mathbb{R}^2)$ and $W^{m,p}(\mo; \mathbb{R}^2)$ which are $2\pi$-periodic in each direction $0x_i$, $i=1,2$, see for example \cite{PC+CF-BlueBook}, are denoted by $\el^p(\mo)$ and $\w^{m,p}(\mo)$, respectively.  
	We simply write $\mathbf{L}^p$ (resp. $\w^{m,p}$) instead of $\el^p(\mo)$ (resp. $\w^{m,p}(\mo)$) when there is no risk of ambiguity. We will also use the notation $\h^m:=\w^{m,2}$. For non integer $r>0$ the Sobolev space $\h^r$ is defined by using classical interpolation method.
	The space $[\mathcal{C}_{\text{per}}^{\infty }(\mathbb{R}^2)]^2:= \mathcal{C}_{\text{per}}^\infty(\mathbb{R}^2,\mathbb{R}^2)$ denotes the space of functions  which are infinitely differentiable and $2\pi$-periodic in each direction $0x_i$, $i=1,2$. 
	
	We also introduce the following spaces 
	\begin{align*}
	\h& =\biggl\{ u \in \el^2(\mo); \int_{\mo} u(x) dx=0,\;\; \mathrm{div}\, u =0 \biggr\},\\
	\ve& =\h^1 \cap \h.
	\end{align*}%
	It is well-known, see \cite{TEMAMA}, that $\h$ and $\ve$ are the closure of
	\begin{equation*}
	\mathcal{V}=\biggl\{u \in [\mathcal{C}_{\text{per}}^\infty(\mathbb{R}^2)]^2; \int_{\mo} u(x)dx=0,\;\; \mathrm{div}\, u =0 \biggr\},
	\end{equation*}
	with respect to the $\el^2$ and $\h^1$ norms.
	We denote by $(\cdot ,\cdot )$ and $|\cdot |$ the inner product and the norm
	induced by the inner product and the norm in $\el^{2}(\CO)$ on ${%
		\h}$, respectively. Thanks to the Poincar\'e inequality we can endow the space $\ve$ with the norm $\lVert u \rVert= \lvert \nabla u \rvert, u\in \ve$.
	
	Let $\Pi:{\el}^2(%
	\mo)\rightarrow \h$ be the Helmholtz-Leray projection, and $%
	\rA=-\Pi \Delta$ be the Stokes operator with the domain $D(\rA)=\h^2(%
	\mo)\cap \h$. It is well-known that $\rA$ is a self-adjoint positive operator with compact inverse, see for instance \cite[Chapter 1, Section 2.6]{TEMAMA}.
	Hence, it has an orthonormal sequence of eigenvectors $\{e_j;\; j\in \mathbb{N}\}$ with corresponding eigenvalues $0<\lambda_1<\lambda_2<...$ The domain of $\rA^r$, $r\in \mathbb{R}$ is characterized by 
	\begin{equation*}
	D(\rA^r)= \ve \cap \h^{2r}, 
	\end{equation*}
	see \cite[page 43]{PC+CF-BlueBook}.  
	
	For $\alpha \in (0,1)$ we set 
	\begin{equation*}
	\lVert u \rVert_\alpha=\sqrt{  \lvert u \rvert^2 +\alpha^2 \lvert \rA u \rvert^2},\; u\in \ve. 
	\end{equation*}
Then, we observe  that $\lVert \cdot \rVert$, $\lVert \cdot \rVert_\alpha,$ $\alpha\in (0,1)$, and $\rvert \rA^\frac12 \cdot \rvert$ define three equivalent norms on $\ve$.
}

For the time being we assume that the stochastic perturbation $G(\bu^\alpha) dW$ is a divergence free function. Then, 
when projecting the system \eqref{Eq: LANS-ALPHA} onto the space of divergence free functions we obtain the following stochastic evolution equation
\begin{align}
\left\{
\begin{array}{ll}
d\bv^\alpha + [\rA\bv^\alpha + \tilde{B}(\bu^\alpha,\bv^\alpha)]dt = \alpha^\frac12 G(\bu^\alpha)dW &\\
\bv^\alpha = \bu^\alpha + \alpha^2\rA\bu^\alpha&\\
\bu^\alpha(t=0) = \xi.
\end{array}
\right. \label{Eq: ABS-LANS-ALPHA}
\end{align}

In a similar way, we can also write the $2D$ Navier-Stokes equations as the abstract evolution equation
\begin{align}
\left\{ 
\begin{array}{ll}
\frac{d}{dt}\bu + \rA\bu +B(\bu,\bu) = 0&\\
\bu(t=0) = \xi
\end{array}
\right. \label{Eq: ABS-NSE}
\end{align}
In \eqref{Eq: ABS-LANS-ALPHA} and \eqref{Eq: ABS-NSE} the nonlinear terms $\tilde{B}$ and $B$ are roughly defined by
\begin{align*}
\tilde{B}(u,v) =&  \Pi(u\cdot\nabla v + \sum_{j=1}^2v_j\nabla u_j) \\
B(u,v) = & \Pi(u\cdot\nabla v),
\end{align*}
respectively. These nonlinear maps satisfy several properties that will be recalled in the last subsection of this section. 

By introducing the following nonlinear maps
\begin{align*}
\tilde{B}_\alpha(u,v) &= (I+\alpha^2 \rA)^{-1}\tilde{B}(u,v)\\
\G (u) &= (I+\alpha^2 \rA)^{-1} G(u),
\end{align*}
the equation \eqref{Eq: ABS-LANS-ALPHA} can be rewritten in the following form:

\begin{align*}
\left\{\begin{array}{ll}
d\bu^\alpha+\rA u^\alpha + \tilde{B}_\alpha(\bu^\alpha,\bv^\alpha) = \alpha^\frac12 \G (u^\alpha)dW &\\
\bu^\alpha(t=0)= \xi.&
\end{array} 
\right.
\end{align*}

In the next few lines we will introduce an abstract stochastic evolution equation which will enable us to give a unifying approach to the large and moderate deviations for the problem \eqref{Eq: ABS-LANS-ALPHA}. 
For this purpose we fix a $\delta\in\{0,1\}$ and introduce the function $\lambda_\delta: (0, \infty) \to (0,\infty)$ defined by 
\begin{align*}
\lambda_\delta(\alpha)= \left\{ \begin{array}{ll}
1, &\text{if } \delta = 0,\\ \alpha^\frac12\lambda(\alpha),& \text{if } \delta = 1,
\end{array} \right.
\end{align*}
where $\lambda: (0, \infty) \to (0,\infty)$ is a function satisfying \eqref{Eq:Def-Lambda}. 

	\begin{Rk}\label{Rem:Lambdaalpha}
		In view of the definition of $\lambda_\delta(\alpha)$, we see that as $\alpha \to 0$
		\begin{equation*}
		\lambda_\delta(\alpha) \to 1-\delta.
		\end{equation*}
		Observe also that 	for $\ell\in \{1,2\}$ and $ k \ge \frac{\ell}{2} $
		\begin{equation*}
		\alpha^k \lambda^{-\ell}_\delta (\alpha) \to 0 \text{ as } \alpha \to 0.
		\end{equation*}
		Hence, we can and will assume that  for $\ell\in \{1,2\}$, $ k \ge \frac{\ell}{2} $ and $\alpha \in (0,1)$
		\begin{equation}
		\alpha^k  \lambda^{-\ell}_\delta (\alpha) \le 2
		\end{equation}
		where $\lambda(\alpha)$ is the function considered in \eqref{Eq:Def-Lambda}. 
	\end{Rk}
Before proceeding further we recall the following result on the 2D Navier-Stokees equations, see, for instance,  \cite{PC+CF-BlueBook} and \cite{TEMAMA} for its proof.
\begin{Thm}
	Let $\xi \in \ve$. Then, the problem \eqref{Eq: ABS-NSE} has a unique solution $\bu \in C([0,T]; \ve)\cap L^2(0,T; D(\rA))$.
\end{Thm}
Throughout this paper, the symbol $\bu$ will denote the unique solution of the problem \eqref{Eq: ABS-NSE}. 

Now, we consider the following stochastic evolution equations.

\begin{align}
\left\{ 
\begin{array}{ll}
d\ydela + \left[\rA\ydela + \lambda_\delta(\alpha)\tilde{B}_\alpha(\ydela,\zdela) + \delta\left[\tilde{B}_\alpha(\bu,\zdela) +\tilde{B}_\alpha(\ydela,\valph)\right] \right] &\\
\qquad = -\lambda_\delta^{-1}(\alpha)\delta[\tilde{B}_\alpha(\bu,\valph)-B(\bu,\bu)]dt + \alpha^\frac12\lambda_\delta^{-1}(\alpha) \G (\delta\bu+\lambda_\delta(\alpha)\ydela)dW,&\\
\zdela = \ydela + \alpha^2\rA\ydela,&\\
\ydela(t=0) = (1-\delta)\xi, &
\end{array}
\right. \label{Eq: ABS-SVE}
\end{align}
where $J_\alpha =( I+\alpha^2\rA)^{-1}.$
\begin{Rk}
Observe that, if one is able to prove a LDP result for \eqref{Eq: ABS-SVE} then one just proved LDP and MDP results for the   \eqref{Eq: ABS-LANS-ALPHA}. In fact, the \eqref{Eq: ABS-SVE} reduces to \eqref{Eq: ABS-LANS-ALPHA} when $\delta=0.$
When $\delta=1$ then an LDP result for \eqref{Eq: ABS-SVE} yields an LDP for the process ${\bf y}^{\alpha,1} = \frac{\bu ^\alpha-\bu}{\lambda_\delta(\alpha)}$. This is just an MDP result for the \eqref{Eq: ABS-LANS-ALPHA}.
\end{Rk}
\subsection{Several key estimates}
To close the present section, we  recall and prove several well-known properties of the bilinear maps $B$ and $\tilde{B}$. These properties will play an important role in the sequel.

We first recall the following lemma  that was proved in \cite{FOIAS-HOLM-TITI}.
\begin{Lem}~~
	\begin{enumerate}
		\item Let $X$ be either $B$ or $\tilde{B}$. Then, the operator $%
		X $ can be extended continuously from $\ve\times \ve$ with values in $\ve^{\prime }
		$ (the dual space of $\ve$).
		In particular, for $u,v,w\in \ve$,
		\begin{equation}  \label{2.4}
		|\langle X(u,v),w\rangle_{\ve^{\prime }}|\le c|u|^{1/2}||u||^{1/2}||v||||w||,
		\end{equation}
		and 
		\begin{equation} \label{Eq:Def-Of-tildeB}
		({\tilde{B}}{(}{u}{,}{v}{)},w)=({B}{(}{u}{,}{v}{)},w)-({%
			B}{(}{w}{,}{v}{)},u).
		\end{equation}
		Moreover
		\begin{equation}  
		({B}{(}{u}{,}{v}{)},w)=-({B}{(}{u}{,}{w}{)},v),
		\,\,u,v,w\in \ve, \label{Eq:skew-symmetry}
		\end{equation}
		which in its turn implies that
		\begin{equation}  \label{Eq:Cancel-Prop-B}
		({B}{(}{u}{,}{v}{)},v)=0,\,\,u,v\in \ve.
		\end{equation}
		Also,
		\begin{equation} \label{Eq:Skew-Symm-tildeB}
		({\tilde{B}}{(}{u}{,}{v}{)},w)=({B}{(}{u}{,}{v}{)},w)-({%
			B}{(}{w}{,}{v}{)},u),\,\, u,v,w\in \ve,
		\end{equation}
		and hence
		\begin{equation} \label{Eq:Cancel-Prop-tildeB}
		({\tilde{B}}{(}{u}{,}{v}{)},u)=0, \,\,\, u,v\in \ve.
		\end{equation}	
		\item Furthermore, let $u\in D(\rA), v\in \ve, w\in \h$ and let $X$ be either $%
		B$ or $\tilde{B}$, then
		\begin{equation}  \label{2.9}
		|(X(u,v),w)|\le c \lvert \rA^\frac12 u \rvert^{1/2}|\rA u|^{1/2}\lvert \rA^\frac12 v \rvert|w|.
		\end{equation}
		
		\item Let $u\in \ve, v\in D(\rA), w\in \h$, then
		\begin{equation}  \label{2.10}
		|({B}{(}{u}{,}{v}{)},w)|\le c \lvert \rA^\frac12 u \rvert\lvert \rA^\frac12 v \rvert^{1/2}|\rA v|^{1/2}|w|.
		\end{equation}
		\item The operator $B$ and $\tilde{B}$ can be also extended continuously from $D(\rA^\frac12) \times \h$  with
		values in $D(\rA^{-1} )$. In particular,   if  $u\in D(\rA^\frac12), v\in \h, w\in D(\rA)$, then  
		\begin{equation}\label{Eq:Est-Ext-tildeB}
		\langle \tilde{B}(u,v), w\rangle \le C  [\lvert u \rvert^\frac12 \lvert \rA^\frac12 u\rvert^\frac12 \lvert \rA^\frac12 w \rvert^\frac12 \lvert \rA w \rvert^\frac12+ \lvert \rA w \rvert \,\lvert \rA^\frac12 u \rvert] \lvert v \rvert,
		\end{equation}
		hence the Poincar\'e inequality yields
		\begin{equation}  \label{2.15}
		|\langle{\tilde{B}}{(}{u}{,}{v}{)},w\rangle_{D(\rA)^{\prime }}|\le C \lvert \rA^\frac12 u \rvert |v| |\rA w|.
		\end{equation}
		Also, by symmetry we have for all $u\in D(\rA), v\in \h, w\in D(\rA^\frac12)$,
		\begin{equation}  \label{2.15-B}
		|\langle{\tilde{B}}{(}{u}{,}{v}{)},w\rangle_{D(\rA)^{\prime }}|\le C \lvert \rA u \rvert |v| |\rA^\frac12 w|.
		\end{equation}
	\end{enumerate}
\end{Lem}
\begin{Rem}
	Observe that by using the H\"older and the Gagliardo-Nirenberg inequalities one can refine the estimate \eqref{2.10} as follows.
	 Let $u\in \ve, v\in D(\rA), w\in \h$, then
	\begin{equation}  \label{refine-2.10}
	|({B}{(}{u}{,}{v}{)},w)|\le c \lvert u \rvert^\frac12 \lvert \rA^\frac12 u \rvert^\frac12 \lvert \rA^\frac12 v \rvert^{1/2}|\rA v|^{1/2}|w|,
	\end{equation}
	see \cite{PC+CF-BlueBook} and \cite{TEMAMA} for the detail.
	
	Notice also that thanks to \eqref{Eq:Cancel-Prop-tildeB} we have 
 	\begin{equation}\label{Eq:Cancel-Balpha}
	(\tilde{B}_\alpha(u,v), J_\alpha^{-1}u) = \langle \tilde{B}(u,v), u\rangle_{\ve^\prime} =0 \text{ for all } u,\, v\in \ve.
	\end{equation}

\end{Rem}
 
The following lemma, which was proved in \cite{CAO-TITI}, will be needed in several places later on. 
\begin{Lem}
	Let $\phi \in \h$, $w \in D(\rA^\frac12)$. Then, for any $\alpha \in (0,1)$ we have 
	\begin{equation}\label{Eq:EstfromCao-Titi-}
	\langle \phi -J_\alpha \phi, w \rangle \le \frac{\alpha}{2} \lvert \phi \rvert \lvert \rA^\frac12 w \rvert.  
	\end{equation}
\end{Lem}
We also need the following three lemmata.
\begin{Lem}
	There exists a constant $C>0$ such that for any $\alpha \in (0,1)$, any $y, u\in D(\rA)$ we have 
	\begin{align}
	&	( J_\alpha^{-1}y, B(u,u)-\tilde{B}_\alpha (u, J^{-1}_\alpha u) )  \le C \left[\frac\alpha2 \lvert \rA^\frac12 y \rvert + \alpha^2 \lvert \rA y \rvert \right] \left(\lvert B(u,u)\rvert + \lvert \rA^\frac12 u \rvert\, \lvert\rA u \rvert \right) , \label{Eq:Est-DiffBandtildeB}\\
	&	( y, B(u, u)-\tilde{B}_\alpha (u, J^{-1}_\alpha u) ) \le C \frac\alpha2 \lvert \rA^\frac12 y \rvert  \lvert B(u, u)\rvert +\frac\alpha2 C\lvert\rA u \rvert^2 \,\lvert y \rvert. \label{Eq:Est-DiffBandtildeB-2}
	\end{align}
\end{Lem}
\begin{proof}
	Let $y, u \in D(\rA)$ and $\alpha \in (0,1)$. In order to simplify the notation we set $\phi = B(u, u)$. By the bilinearity of $\tilde{B}$ and $B$, and the fact $B(u, u)=\tilde{B}(u, u)$ we have 
	\begin{align*}
	&	( J_\alpha^{-1}y , B(u, u)-\tilde{B}_\alpha (u, J^{-1}_\alpha u) ) \\
	=& ( J^{-1}_\alpha y, B(u, u)-J_\alpha B(u, u)) +  \alpha^2  (J_\alpha^{-1} y, B(u, u)-J_\alpha \tilde{B}(u,\rA u)) ,\\
	=& (y, B(u, u)-J_\alpha B(u, u)) +\alpha^2  (\rA y, B(u, u)-J_\alpha \tilde{B}(u,\rA u) )  +\alpha^2 ( J_\alpha^{-1} y, B(u, u)-J_\alpha \tilde{B}(u,\rA u) ).
	\end{align*}
	By using the last line, \eqref{Eq:EstfromCao-Titi-}, the facts that 
	\begin{align}
	\lvert \alpha^2  \rA (I+\alpha^2 \rA)^{-1}\rvert_{\mathscr{L}(H)}= \sup_{k \in \mathbb{N}} \frac{\alpha^2\lambda_k}{1+\alpha^2 \lambda_k}\le 1, \\
	\lvert \alpha  \rA^\frac12 (I+\alpha^2 \rA)^{-1}\rvert_{\mathscr{L}(H)}= \sup_{k \in \mathbb{N}} \frac{(\alpha^2\lambda_k)^\frac12}{1+\alpha^2 \lambda_k}\le \frac12, \label{Eq:3.20}
	\end{align}
	and the inequalities \eqref{Eq:Est-Ext-tildeB}  we easily establish \eqref{Eq:Est-DiffBandtildeB}. 
	
	The second estimate \eqref{Eq:Est-DiffBandtildeB-2} is proved in  a similar way. 
 
\end{proof}
\begin{Lem}\label{lemma3.4}
	There exists a constant $C>0$ such that for any $y, \in \h, v, w \in D(\rA),$ and any $\alpha \in (0,1)$ we have 
	\begin{equation}\label{Eq:linear-pert}
	 ( \tilde{B}_\alpha(v, J_\alpha^{-1}w), y )\le C \lvert \rA v \rvert \,\lvert y\rvert\left( \lvert \rA^\frac12 w \rvert +\alpha \lvert \rA w \rvert\right).
	\end{equation}
\end{Lem}
\begin{proof}
	Throughout this proof $C$ will denote a constant independent of $\alpha$.
	
	Let $y\in \h, v, w \in D(\rA),$ and  $\alpha \in (0,1)$. Observe that 
	\begin{equation*}
	(\tilde{B}_\alpha(v, J_\alpha^{-1}w), y )= (\tilde{B}(v,w), J_\alpha y)  +\alpha^2 \langle \tilde{B}(v,\rA w), J_\alpha y \rangle_{D(\rA)'}.
	\end{equation*}
	Now, by applying the inequalities \eqref{Eq:Skew-Symm-tildeB}, \eqref{Eq:Def-Of-tildeB}, \eqref{Eq:Est-Ext-tildeB} and the H\"older inequality we find that there is a constant $C>0$ such that 
	\begin{align*}
	  (\tilde{B}_\alpha(v, J_\alpha^{-1}w), y )\le C \lvert \rA v \rvert \lvert \rA^\frac12 w \rvert \lvert J_\alpha y \rvert + \alpha^2 \lvert \rA^\frac12 J_\alpha y \rvert \lvert \rA w \rvert \lvert \rA v \rvert.
	\end{align*}
	By using the fact $\lvert (\alpha^2 \rA)^\frac12 (I+\alpha^2 \rA)^{-1}\rvert_{\mathscr{L}(\h)}\le \frac12$ and $\lvert J_\alpha \rvert_{\mathscr{L}(\h)}\le 1$ we see that 
	\begin{align*}
	  (\tilde{B}_\alpha(v, J_\alpha^{-1}w), y) \le C  \lvert \rA v \rvert \lvert \rA^\frac12 w \rvert \lvert  y \rvert + \alpha \lvert  y \rvert \lvert \rA w \rvert \lvert \rA v \rvert,
	\end{align*}
	which completes the proof of the lemma.
\end{proof}

\begin{Lem}\label{Lem:Need-Uniqueness}
	There exists a constant $C>0$ such that for any $\alpha \in (0,1)$ and any $u\in D(\rA),\,y\in D(\rA):$
	\begin{align*}
	\lqq{\left|(\tilde{B}_\alpha(u,J_\alpha^{-1} y ),J_\alpha^{-1} y )  +(\tilde{B}_\alpha(y,\valph),J_\alpha^{-1} y )   \right|} \\
	&\leq\frac14 |y|^2+\frac{\alpha^2}{2}|\rA^\frac12y|^2 + \frac{\alpha^4}{4}|\rA y|^2+\alpha^{-2}C|u|^2_{D(\rA)}\left[ |y|^2+\alpha^2|\rA^\frac12y|^2 \right].
	\end{align*}
\end{Lem}
\begin{proof}
	Let $u\in D(\rA),\,y\in D(\rA)$ and $z=y+\alpha^2\rA y$.
	Using equation (4) on page 5 of \cite{FOIAS-HOLM-TITI} we obtain:
	\begin{align*}
	 (\tilde{B}_\alpha(u,z),z) &= ( J_\alpha\tilde{B}(u,z), J_\alpha^{-1}y  ) \\
	&= \langle \tilde{B}(u,z),y \rangle_{\ve'} \\
	&=\langle  B(u,z),y \rangle_{\ve'} - \langle  B(y,z),u \rangle_{\ve'}.
	\end{align*}
	
	Using the well-known property
	\begin{align*}
	\langle  B(u,z),y \rangle_{\ve'} = - \langle  B(u,y),z \rangle_{\ve'} ,
	\end{align*}
	we obtain
	\begin{align*}
	(\tilde{B}_\alpha(u,z),z )= - \langle  B(u,y),z \rangle_{\ve'} +   \langle  B(y,u),z \rangle_{\ve'}.
	\end{align*}
	Now the H\"older and Young inequalities along with the Sobolev embedding, $D(\rA)\subset \el^\infty$ and $\ve\subset \el^4$ we find that there exists $C>0$ such that 
	\begin{align*}
	(\tilde{B}_\alpha(u,z),z ) &\leq  C|z|\left[ |u|_{\el^\infty}\|y\|+\|y\||u|_{D(\rA)} \right] \\
	&\leq C|z||u|_{D(\rA)}\|y\| \\
	&\leq \frac14 |z|^2 + C|u|^2_{D(\rA)}\|y\|^2\\
	&\leq \frac14 |y + \alpha^2\rA y|^2 +\alpha^{-2}C|u|^2_{D(\rA)}\alpha^2\|y\|^2 \\
	&\leq \frac14 |y|^2 +\frac{\alpha^2}{2} |\rA^\frac12y|^2 + \frac14 |\rA y|^2 + \alpha^{-2}C|u|^2_{D(\rA)}\left[|y|^2+\alpha^2\|y\|^2 \right].
	\end{align*}
 
	The last line of the inequality completes the proof of the lemma because
	\begin{align*}
	 (\tilde{B}_\alpha(y,J_\alpha^{-1} u  ),z )
	&=( J_\alpha \tilde{B} (y,J_\alpha^{-1} u ),J^{-1}_\alpha y)  \\
	&= \langle \tilde{B} (y,J_\alpha^{-1} u), y \rangle_{D(\rA)'} =0.
	\end{align*}
\end{proof}

\section{Assumptions on the noise coefficient and a well-posdeness result}\label{Assum-Noise}
 
 This section is devoted to the formulation of the standing assumption on the noise and the presentation of a well-posedness result. 
\subsection{Formulation of the assumptions on the noise}
Throughout we fix  a complete filtered probability space $\mathscr{U}:=(\Omega, \mathscr{F}, \mathbb{F}, \mathbb{P})$  where the filtration $\mathbb{F}=\{\mathscr{F}_t;\;\; t\in [0, T] \}$ satisfies the usual conditions. We also fix two separable Hilbert spaces  $\mathrm{K}$ and $\mathrm{K}_1$ such that the canonical injection $\iota: \mathrm{K} \to \mathrm{K}_1$ is Hilbert-Schmidt. The operator $Q=\iota \iota^\ast$, where $\iota^\ast$ is the adjoint of $\iota$, is symmetric, nonnegative. Since $\iota$ is Hilbert-Schmidt $Q$ is also of trace class. Moreover, from \cite[Corollary C.0.6]{Prevot+Rockner} we infer that $\mathrm{K}= Q^{\frac12}(\mathrm{K}_1) $. Now, let $W$ be a cylindrical Wiener process evolving on $\mathrm{K}$. 
It is well-known, see \cite[Theorem 4.5]{DP+JZ-14}, that $W$ has the following series representation 
$$ W(t)= \sum_{j=1}^\infty \sqrt{q_j} \beta_j(t) h_j, \quad t\in [0,T],$$ 
where $\{\beta_j; \; j \in \mathbb{N} \}$ is a sequence of mutually independent and identically distributed standard Brownian motions,  $\{h_j; j\in \mathbb{N} \}$ is an orthonormal basis of $\mathrm{K}$ consisting of eigenvectors of $Q$ and $\{q_j;\; j\in \mathbb{N}\}$ is the family of eigenvalues of $Q$. It is also well-known, see \cite[Section 4.1]{DP+JZ-14} and \cite[Section 2.5.1]{Prevot+Rockner}, that $W$ is an $\mathrm{K}_1$-valued Wiener process  with covariance $Q$. 

Now, we recall few basic facts about stochastic integrals with respect to a cylindrical Wiener process evolving on $\mathrm{K}$.  For this purpose, let $\mathscr{H}$ be a separable Banach space, $\mathscr{L}(\mathrm{K}, \mathscr{H})$ the space of all bounded linear $\mathscr{H}$-valued  operators defined on $\mathrm{K}$, and  $\mathscr{M}_T^2(\mathscr{H}):=\mathscr{M}^2(\Omega\times [0,T]; \mathscr{H})$ the space of all equivalence classes of $\mathbb{F}$-progressively measurable processes $\Psi: \Omega\times [0,T]\to\mathrm{K}$ satisfying 
$$ \E\int_0^T \Vert \Psi(r)\Vert^2_{\mathscr{H}}dr <\infty.$$ 
We denote by $\mathscr{L}_2(\mathrm{K}, \mathscr{H})$ the Hilbert space of all operators $\Psi\in \mathscr{L}( \mathrm{K}, \mathscr{H} )$ satisfying 
$$ \Vert \Psi\Vert_{\mathscr{L}_2(\mathrm{K}, \mathscr{H})}^2 =\sum_{j=1}^\infty \Vert \Psi h_j\Vert^2_{\mathscr{H}}<\infty. $$   

From the theory of stochastic integration on infinite dimensional Hilbert space, see \cite[Chapter 5, Section 26 ]{MET-82} and \cite[Chapter 4]{DP+JZ-14}, for any $\Psi \in \mathscr{M}^2_T(\mathscr{L}_2(\mathrm{K}, \mathscr{H}))$ the process $M$ defined by 
$$ M(t) =\int_0^t \Psi(r)dW(r), t\in [0,T],$$ is a $\mathscr{H}$-valued martingale. Moreover,  we have the following It\^o isometry 
\begin{equation}
	\E \biggl(\biggl\Vert \int_0^t \Psi(r) dW(r)\biggr\Vert^2_{\mathscr{H}}\biggr) =\E \biggl(\int_0^t \Vert \Psi(r)  \Vert^2_{\mathscr{L}_2(\mathrm{K}, \mathscr{H})} dr \biggr), \quad \forall t\in [0,T],
\end{equation}	
and the Burkholder-Davis-Gundy's (BDG's) inequality 
\begin{equation}
	\E\biggl(\sup_{0\le s\le t}\biggl \Vert \int_0^s \Psi(r) dW(r)\biggr\Vert^q\biggr) \le C_q \E \biggl( \int_0^t \Vert \Psi(r) \Vert^2_{(\mathrm{K}, \mathscr{L}_2(\mathscr{H})}dr \biggr)^\frac q2, \forall t\in [0,T], \forall q\in (1,\infty).
\end{equation}

The standing assumptions on $G$ are given bellow.

\begin{assum} \label{Assum-Uniqness}
	The map $G: \h \longrightarrow \mathscr{L}_2 (\mathrm{K}, \h)\cap \mathscr{L}_2 (\mathrm{K}, \ve)$ satisfies the following: there exists $C>0$ such that   for any $u, v \in \h$
	\begin{align}\label{assum:LipschitsG1}
\| G (u)  - G(v)  \|_{\mathscr{L}_2 ( \rK, \ve)}		+ \| G (u)  - G(v)  \|_{\mathscr{L}_2 ( \rK, \h)} & \leq C | u - v | \notag \\
		\| G(u)\|_{\mathscr{L}_2 (\rK, \ve)} +	\| G(u)\|_{\mathscr{L}_2 (\rK, \h)} & \leq C (1 + \lvert u \rvert). \notag 
	\end{align}
\end{assum}

\begin{Rk} \label{Rem:uniqueness}
	From the above assumption, we infer that there exists $C>0$ such that   for any $u, v \in \ve$
	\begin{displaymath}
		\| G (u)  - G(v)  \|_{\mathscr{L}_2 (\rK, \ve)} \leq \;
		\left \lbrace \begin{array}{cccc}
			\hspace{-2.4cm}C \lvert \rA^{\frac{1}{2}} \left(u - v\right)\rvert & 	 \\
			\hspace{0.15cm} C \lvert u - v\rvert  ~    + ~ \alpha^2   \lvert \rA^{\frac{1}{2}} \left(u  - v\right)\rvert, 	 &
		\end{array}\right.
	\end{displaymath}
 and 
	\begin{displaymath}
		\| G (u)  \|_{\mathscr{L}_2 (\rK, \ve)} \leq 
		\left \lbrace \begin{array}{cccc}
			\hspace{-2.9cm}C \left( 1 + \lvert u \rvert \right) & \\
			\hspace{-2.3cm} C \left(1 + \lvert \rA^{\frac{1}{2}}  u \rvert\right) & 	 \\
			\hspace{0.19cm} C \left(1 + \lvert u \rvert\right)  ~  + ~ \alpha^2   \lvert \rA^{\frac{1}{2}}  u\rvert.  	 & 
		\end{array}\right.
	\end{displaymath}
\end{Rk}

\subsection{The well-posedness of the basic problem \eqref{Eq: ABS-SVE}}
In this subsection we state a well-posedness result of basic problem \eqref{Eq: ABS-SVE}. Since there are several papers which deal with the existence of solutions of LANS-$\alpha$ model we give a rather sketchy proof of this well-posedness result.

	We first give 
	the concept of solutions to \eqref{Eq: ABS-SVE} that we adopt in this paper. 
	
	\begin{Def}\label{Def: Def-Solution}
		Given $\delta\in\{0,1\},\,u_0\in \ve:=D(\rA^\frac12)$, a stochastic process $\ydela:[0,T]\to \ve$ is a strong solution to \eqref{Eq: ABS-SVE} if and only if
		\begin{itemize}
			\item $\ydela$ is $\BF-$adapted, i.e., for each $t$, $\ydela(t)$ is $\mathscr{F}_t$-measurable,
			\item  and $\ydela\in C([0,T];\ve)\cap \el^2([0,T];D(\rA)) $ with probability 1.
			\item Moreover, for all $t\in[0,T]$, $\BP$.a.s.
			\begin{align*}
			\lqq{\langle \ydela(t),\varphi\rangle + \int_0^t\langle \rA\ydela+\lambda_\delta(\alpha)\tilde{B}_\alpha(\ydela,\zdela) + \delta[\tilde{B}_\alpha(\bu,\zdela) + \tilde{B}_\alpha(\ydela,\valph),\varphi] \rangle ds} \\
			&= \langle\ydela(0),\varphi  \rangle + \int_0^t \langle\lambda_\delta^{-1}(\alpha) [B(\bu,\bu) - \tilde{B}_\alpha(u,\valph)] ,\varphi \rangle ds \\
			& \quad + \alpha^\frac12 \lambda_\delta^{-1}(\alpha)\langle \int_0^t \G (\delta \bu + \lambda_\delta(\alpha)\ydela)dW,\varphi  \rangle
			\end{align*}
		\end{itemize}
	\end{Def}
With this definition in mind we now give the following results.

\begin{Prop}
		Let $\delta\in\{0,1\}$ and $\xi\in \ve.$ Assume that $G$ satisfies the Assumption \ref{Assum-Uniqness}.  If the stochastic evolution equation \eqref{Eq: ABS-SVE}  has two strong solutions $\ydela_i,\, i=1,2$ in the sense of Definition \ref{Def: Def-Solution} such that 
		\begin{align*}
			\ydela_i \in \el^p(\Omega;C[0,T];\ve)\cap \el^2([0,T];D(\rA)) \text{ for all $p\in [1,\infty],$},
		\end{align*}
		then with probability 1
		$$\ydela_1 (t)=\ydela_2(t) \text{ for all } t\in [0,T].$$
		That is, the strong solution to the system \eqref{Eq: ABS-SVE} is pathwise unique. 
\end{Prop}
\begin{proof}
	The proof of the proposition follows the same lines as in \cite{CARABALLO1},  but for the sake of completeness we give the detail.
	
Since the system reduces to the $2D$ Stochastic LANS-$\alpha$ model when $\delta=0$ and the proof of the uniqueness result can be done as in   \cite{CARABALLO1}. Hence, we will focus on the case $\delta = 1.$ 
Since the parameter $\alpha$ does not play an important role for the proof of uniqueness we can and will assume $\alpha =1$ and $\lambda_\delta(\alpha) = 1$.
With these in mind, let $\mathbf{y}_1$ and $\mathbf{y}_2 $ be two solutions the system \eqref{DCNS-alpha-delta}.
We set 
\begin{align*}
	\mathbf{y} &= \mathbf{y}_1 - \mathbf{y}_2\\
	\mathbf{z} &= \mathbf{y} + \rA\mathbf{y}\\
	\mathbf{z}_i&=\mathbf{y}_i+\rA\mathbf{y}_i, \quad \text{for }   i = 1, 2.
\end{align*}
Then, using the bilinarity of $\tilde{B}_\alpha,\, B$ we see that $\mathbf{y}$ satisfies:
\begin{align*} 
	\left\lbrace
	\begin{array}{ll}
		\lqq{d\mathbf{y} +\left[\rA\mathbf{y} + \tilde{B}_1(\mathbf{y},\mathbf{z}_1)+\tilde{B}_1(\mathbf{y}_2,\mathbf{z}) + \tilde{B}_1(\mathbf{u},\mathbf{z})+\tilde{B}_1(\mathbf{y},J_1^{-1}\mathbf{u})\right]} \\
		&= \left[\mathbf{G}_1(\mathbf{u}+\mathbf{y}_1)-\mathbf{G}_1(\mathbf{u}+\mathbf{y}_2)\right]dW\\
		\mathbf{y}(t=0) &= 0.
	\end{array} \right.
\end{align*}
Recall that $\tilde{B}_1(x,y) = (I+A)^{-1}\tilde{B}(x,y) $ and $\mathbf{G}_1(x) = (I+A)^{-1}G(x)$.
Let $N>0$ and $\tau_N$ be the stopping time defined by
\begin{align*}
	\tau_N=\inf\left\{t\in[0,T]: |\rA^\frac12\mathbf{y}_1(t)|>N \right\} \wedge \inf\left\{t\in[0,T]: |\rA^\frac12\mathbf{y}_2(t)|>N \right\}.
\end{align*}
Let $t\in(0,T]$ be fixed. 
 
Applying the It\^o formula to $\mathbf{y}$ and the functional $\varphi(x) = |x|^2 + |\rA^\frac12 x|^2 $ for $x\in\ve$, yields
\begin{align*}
	\lqq{d\left( |\mathbf{y}|^2+|\rA^\frac12\mathbf{y}|^2 \right) + \left[|\rA^\frac12\mathbf{y}|^2+|\rA\mathbf{y}|^2 +( \mathbf{y},\tilde{B}(\mathbf{y}_2,\mathbf{z}) + \tilde{B}(\mathbf{u},\mathbf{z}) ) \right]dt}\\
	=&\Vert G(\mathbf{u}+\mathbf{y}_1)-G(\mathbf{u}+\mathbf{y}_2)\Vert^2_{\mathscr{L}_2(\rK,\h)}dt + \langle \mathbf{y},G(\mathbf{u}+\mathbf{y}_1)-G(\mathbf{u}+\mathbf{y}_2)dW \rangle,
\end{align*}
where we used the facts that for $x\in \h$
\begin{align*}
	\varphi'(\mathbf{y})[(I+\rA^{-1})x] &= 2 ( \mathbf{y}+\rA\mathbf{y}, (I+\rA)^{-1}x ) = 2\red{ (\mathbf{y},x) } \\
	\frac12 \varphi''(\mathbf{y})[x,x] &= ( x+\rA x, (I+\rA)^{-1}x ) = |x|^2,
\end{align*}
and the cancellation property \eqref{Eq:Cancel-Prop-tildeB}.
Note that thanks to \eqref{2.15-B}, the continuous embedding $\h\subset D(\rA^\frac12)$ and the Young inequality we deduce that there exists a constant $C_0>0$ such that 
\begin{align}
	\lqq{2|( \tilde{B}(\mathbf{y}_2+\mathbf{u},\mathbf{z}),\mathbf{y} )|} \nonumber\\
	&\leq \left[|\rA^\frac12 \mathbf{y}|^2 + |\rA\mathbf{y}|^2\right] + C_0\left[ |\rA\mathbf{y}_2|^2 + |\rA\mathbf{u}|^2 \right]\left[ |\rA^\frac12\mathbf{y}_2|^2+|\rA\mathbf{u}|^2 + |\mathbf{y}|^2 \right] . \label{Eq:uniq-Ito}
\end{align}
Now, we let 
\begin{align*}
	\Psi(t)  = e^{C_0\int_0^t[|\rA\mathbf{y}_2(s)|^2 + |\rA\mathbf{u}(s)|^2]ds},\quad  t\in [0,T],
\end{align*}
and  apply the It\^o formula to the real-valued process 
\begin{align*}
	x(t) = \Psi(t)\varphi(\mathbf{y}(t)) = \Psi(t)\left[|\mathbf{y}(t)|^2 + |\rA^\frac12\mathbf{y}(t)|^2\right], \; t\in [0,T].
\end{align*}
This procedure along with \eqref{Eq:uniq-Ito} and the Lipschitz continuity of $G$ yield
\begin{align*}
	\lqq{x(t\wedge \tau_N) + 2\int_0^{t\wedge \tau_N}\psi(s)\left[ |\rA^\frac12\mathbf{y}(s)|^2 + |\rA\mathbf{y}(s)|^2 \right]ds} \\
	&\leq x(0) - C_0\int_0^{t\wedge\tau_N}\Psi(s)\left[ |\rA\mathbf{y}_2(s)|^2 + |\rA\mathbf{u}(s)|^2 \right]\left[ |\rA^\frac12\mathbf{y}(s)|^2+|\mathbf{y}(s)|^2\right]ds \\
	& \,{ }+\int_0^{t\wedge\tau_N}\Psi(s)|( \tilde{B}(\mathbf{y}_2+\mathbf{u},\mathbf{z}),\mathbf{y} )|ds \\
	& { }+\int_0^{t\wedge\tau_N}\Psi(s)\Vert G(\mathbf{u}+\mathbf{y}_1)-G(\mathbf{u}+\mathbf{y}_2)  \Vert^2_{\mathscr{L}(\rK,\h)}ds \\
	& { } +\int_0^{t\wedge\tau_N}\Psi(s)\langle \mathbf{y}(s),G(\mathbf{u}+\mathbf{y}_1)-G(\mathbf{u}+\mathbf{y}_2) dW\rangle \\
	&\leq x(0) + C\int_0^{t\wedge\tau_N}\Psi(s)|\mathbf{y}_1-\mathbf{y}_2|_{\h}^2ds  \\
	& { }+ \int_0^{t\wedge\tau_N}\Psi(s)\left[ |\rA^\frac12\mathbf{y}(s)|^2 + |\rA\mathbf{y}(s)|^2 \right]ds \\
	& { }+ \int_0^{t\wedge\tau_N}\Psi(s)\langle \mathbf{y}(s), G(\mathbf{u}+\mathbf{y}_1) - G(\mathbf{u}+\mathbf{y}_2)dW  \rangle. 
\end{align*}
Observe that, 
\begin{align*}
	|\mathbf{y}_1-\mathbf{y}_2|^2 = |\mathbf{y}|^2 \leq |\mathbf{y}|^2 + |\rA^\frac12\mathbf{y}|^2.
\end{align*}
Hence, by taking the mathematical expectation and using the fact that the stopped stochastic integral in the above inequalities are a zero mean martingale we obtain that 
\begin{align*}
	\lqq{\mathbb{E}x(t\wedge\tau_N) + \mathbb{E}\int_0^{t\wedge\tau_N}\Psi(s)\left[ |\rA^\frac12\mathbf{y}(s)|^2 + |\rA\mathbf{y}(s)|^2 \right]ds} \\
	&\leq \mathbb{E}x(0) + C\int_0^{t\wedge\tau_N}\mathbb{E}x(s\wedge\tau_N)ds.
\end{align*}
By applying the Gronwall Lemma and the fact that $x(0) = 0.$
We see that
\begin{align*}
	\mathbb{E}x(t\wedge\tau_N) = 0, \quad t\in[0,T].
\end{align*}
Since $x\geq0$ and $\Psi>0$ we see that for all $t\in[0,T]$ a.e.
\begin{align*}
	\mathbf{y}_1(t) = \mathbf{y}_2(t), \quad\text{in }\quad \ve.
\end{align*}
From the fact $\mathbf{y}_i\in C([0.T];D(\rA^\frac12))$ a.e.
We finally conclude that a.e. for all $t\in[0,T]$
\begin{align*}
	\mathbf{y}_1(t) = \mathbf{y}_2(t),
\end{align*}
which completes the proof of the proposition.
 
\end{proof}

\begin{Thm}\label{THM:Exist-1}
	Let $\delta\in\{0,1\}$ and $\xi\in \ve.$ Assume that $G$ satisfies the Assumption \ref{Assum-Uniqness}. Then the stochastic evolution equation \eqref{Eq: ABS-SVE} has a unique solution $\ydela$ in the sense of Definition \ref{Def: Def-Solution} such that 
\begin{align*}
 \ydela\in \el^p(\Omega;C[0,T];\ve)\cap \el^2([0,T];D(\rA)) \text{ for all $p\in [1,\infty).$} 
\end{align*}
\end{Thm}
\begin{proof}
Observe that if $\delta = 0$, then the problem \eqref{Eq: ABS-SVE} reduces to the stochastic system \eqref{Eq: ABS-LANS-ALPHA}. Under the Assumption \ref{Assum-Uniqness} it was proved in \cite{CARABALLO1} that \eqref{Eq: ABS-LANS-ALPHA} has a unique solution ${\bf y}^{\alpha,0}$ satisfying ${\bf y}^{\alpha,0}\in \el^4(\Omega;C([0,T];\ve))\cap \el^2(0,T;D(\rA))$.
	The fact that ${\bf y}^{\alpha,0}\in \el^p(\Omega;C([0,T];\ve))\cap \el^2(0,T;D(\rA))$ for all $p\geq1$ is proved in \cite{Gabriel}.
	
	Next, we recall that the deterministic evolution equation \eqref{Eq: ABS-NSE} with initial data $\xi\in \ve$ has a unique strong solution, $\bu\in C([0,T];\ve)\cap \el^2(0,T;D(\rA)).$
	Note that $\bu$ is deterministic .
	If $\delta = 1, $ then, as discussed above, the stochastic process ${\bf y}^{\alpha,1} = \frac{{\bf y}^{\alpha,0}-\bu}{\lambda_\delta(\alpha)}\in \el^p(\Omega;C([0,T];\ve))\cap \el^2(0,T;D(\rA))$ satisfies the problem \eqref{Eq: ABS-SVE}.
\end{proof}
\begin{Rem}\label{Rem-Exist-G2}
	The existence and uniqueness of a strong solution to \eqref{Eq: ABS-SVE}  enables us to define a Borel measurable map $\Gamma^{\alpha,\delta}_\xi:C([0,T];\mathrm{K}) \to C([0,T];\h)\cap \el^2(0,T; D(\rA^\frac12)) $ such that $\Gamma^{\alpha,\delta}_\xi(W)$ is the unique solution to \eqref{Eq: ABS-SVE} on the filtered probability space $(\Omega, \mathscr{F}, \mathbb{F}, \mathbb{P})$ with the Wiener process $W$.  
\end{Rem}

\section{Analysis of the  controlled evolution equations}
\label{sec:4 Analysis}
In order to describe the rate functions associated to the LDP and MDP results, we also need to introduce few additional notations and two auxiliary problems: the stochastic and deterministic controlled evolution equations. 

For fixed $M>0$ we set
\[\mathcal{A}_M=\Big\{h \in L^2(0, T; \mathrm{K}): \int_0^T \rVert h(r)\lVert^2_{\mathrm{K}} dr \leq M\Big\}.\]
The set $\mathcal{A}_M$, endowed with the weak topology 
\begin{equation}\label{Eq:Metric-SM}
d_1(h, k)=\sum_{k\geq 1} \frac1{2^k}
\big|\int_0^T \big(h(r)-k(r), \tilde{e}_k(r)\big)_{\mathrm{K}} dr \big|,
\end{equation}
where
$ ( \tilde{e}_k , k\geq 1)$ is an  orthonormal basis for
$L^2(0, T ;  \mathrm{K})$,  is a
Polish  (complete separable metric)  space, see \cite{AB+PD+VM}.

\noindent We also introduce the
class
$\mathscr{A}$ as the  set of $\mathrm{K}$-valued
$(\mathscr{F}_t)-$predictable stochastic processes $h$ such that $\int_0^T
\lVert h(r)\rVert^2_{\mathrm{K}} dr < \infty, \; $ a.s.
For $M>0$ we set
\begin{equation} \label{AM}
\mathscr{A}_M=\{h\in \mathscr{A}: h \in
\mathcal{A}_M \; a.s.\}.
\end{equation} 
\subsection{Analysis of the stochastic controlled evolution equations}
With the above  notations at hand we now consider the stochastic controlled equation:
\begin{align}
& d \mathbf{y}^{\alpha, \delta} + A \mathbf{y}^{\alpha, \delta} dt  +   \lambda_{\delta}  (\alpha) \tilde{B}_{\alpha}(\mathbf{y}^{\alpha,\delta}, \mathbf{z}^{\alpha, \delta}) dt + \delta \tilde{B}_{\alpha}(\bu, \mathbf{z}^{\alpha, \delta}) dt + \delta   \tilde{B}_{\alpha}(\mathbf{y}^{\alpha, \delta},J_\alpha^{-1}\bu) dt   \notag \\ 
& \hspace{0.9cm} +  \delta \lambda_{\delta}^{-1} (\alpha) \left[\tilde{B}_{\alpha}(\bu, J_{\alpha}^{-1} \bu)  - B(\bu, \bu)\right] dt \notag\\
& = \G  (\delta \bu + \lambda_{\delta} (\alpha) \mathbf{y}^{\alpha, \delta} ) h (t) dt + \alpha^{\frac{1}{2}} \lambda_{\delta}^{-1} (\alpha) \G  (\delta \bu + \lambda_{\delta} (\alpha) \mathbf{y}^{\alpha, \delta}) dW, \label{SCNS-alpha-delta}
\end{align}
where $h\in L^2(0,T; \mathrm{K})$.

We now need   to prove the existence and uniqueness of \eqref{SCNS-alpha-delta} and derive uniform estimates for its solution.  This will be the subject of the following theorem.
	\begin{Thm}\label{THM:SCNS-alpha-delta1}
			Let $\delta\in \{0,1\}$, $\xi \in {D(\rA^\frac12)}$, $p\in [1,\infty)$. Let us also fix $M>0$ and $h\in \mathscr{A}_M$. If Assumption \ref{Assum-Uniqness} is satisfied,  then  the stochastic controlled system 	
	\eqref{SCNS-alpha-delta} has a unique solution $\mathbf{y}_{h}^{\alpha, \delta} \in C([0, T]; \ve ) \cap \el^2 (0, T;D(\rA))$ such that 
		$$\mathbf{y}_{h}^{\alpha, \delta}= \Gamma^{\alpha,\delta}_\xi \biggl(W + \alpha^{-\frac12} \lambda_\delta(\alpha )\int_0^{\cdot} h(r)dr  \biggr) .$$
	
		Furthermore, then there exists a constant $ C>0$ (which may depend on $p$) 
 such that for any $\alpha \in (0,1)$ we have 
		\begin{equation} \label{Estim}
\begin{split}
	&  \mathbb{E} \sup_{t \in [0, T]} \Vert \mathbf{y}_{h}^{\alpha, \delta} (t) \Vert_{\alpha}^{2p} +    \mathbb{E} \int_{0}^{t} \Vert \mathbf{y}_{h}^{\alpha, \delta} (t) \Vert_{\alpha}^{2p - 2} \Vert\mathbf{y}_{h}^{\alpha, \delta} \Vert_{\ve}^2 ds  \\
	&  \leq C\left( 1+  \lvert \rA^\frac12  \mathbf{y}_{h}^{\alpha, \delta} (0)\rvert^2 + CM^\frac12 T^\frac12\left(1+ \delta^{2p}\lvert \rA^\frac12 \bu \rvert^{2p} \right)\right)e^{\int_0^T\Phi_\delta (s) ds}\;\mathbb{P}\text{-}a.s.,
\end{split}
	\end{equation}
	where 
	\begin{equation*}
	\Phi_\delta:= 1+ \lVert h \rVert_{\rK}+\delta^2 [\lvert \rA\bu\rvert^2+\lvert \bu\rvert^2 ] + \lvert B(\bu,\bu) \rvert^2+ \lvert \rA \bu \rvert^2\lvert \rA^\frac12 \bu\rvert^2.
	\end{equation*}
\end{Thm}

\begin{proof}
		Let $\delta\in \{0,1\}$, $\xi \in {D(\rA^\frac12)}$, $p\in [1,\infty)$. Let us also fix $M>0$ and $h\in \mathscr{A}_M$.
	The proof of the theorem is divided into two parts. 
	
\noindent 		\textrm{Part I: Well-posedness of problem \eqref{SCNS-alpha-delta}.}
		Since $h \in \mathscr{A}_M$ we have 
	\begin{equation*}
	\mathbb{E}\exp\left(\frac12 \alpha^{-1} \lambda^2_\delta(\alpha)\int_0^T \lVert h(r)\rVert^2 dr \right) <\infty. 
	\end{equation*}
	Thus, by Girsanov's theorem there exists a probability measure $\mathbb{P}_h$ such that 
	\begin{equation*}
	\frac{d\mathbb{P}_h}{d\mathbb{P}}=\exp\left(\frac12 \alpha^{-1}\lambda_\delta(\alpha)^2\int_0^T \lVert h(r)\rVert^2_{\mathrm{K}} dr -\alpha^{-\frac12}\lambda_\delta(\alpha) \int_0^T h(r) dW(r)\right),
	\end{equation*}
	and the stochastic process $\tilde{W}(\cdot):=W(\cdot) +\alpha^{-\frac12}\lambda(\alpha)\int_0^{\cdot} h(r)dr $ defines a cylindrical Wiener process evolving on $\mathscr{H}_0$ and defined  on the filtered probability space $(\Omega, \mathscr{F},\mathbb{F},\mathbb{P}_h)$. We now infer from  Theorem \ref{THM:Exist-1} that on $(\Omega, \mathscr{F},\mathbb{F},\mathbb{P}_h)$ the problem \eqref{SCNS-alpha-delta} with driving noise $\tilde{W}$ admits a unique strong solution $\mathbf{y}_{h}^{\alpha, \delta}$. By Remark \ref{Rem-Exist-G2} we have $\mathbf{y}_{h}^{\alpha, \delta}=\Gamma^{\alpha,\delta}_\xi(\tilde{W})$ on $(\Omega, \mathscr{F},\mathbb{F},\mathbb{P}_h)$ which reads
	$$ \mathbf{y}_{h}^{\alpha, \delta}=\Gamma^{\alpha,\delta}_\xi \left(W(\cdot) +\alpha^{-\frac12}\lambda(\alpha)\int_0^{\cdot} h(r)dr\right) \text{ on } (\Omega, \mathscr{F},\mathbb{F},\mathbb{P}).$$
	
		\noindent \textrm{Part II: Proof of the uniform estimates \eqref{Estim}}
	
The proof of the estimate \eqref{Estim} relies on the application of the It\^o formula to the functional $N_{\alpha}(x):=\lVert x  \rVert_{\alpha}^2$ and the It\^o process $\mathbf{y}^{\alpha, \delta}$. Before proceeding further let us observe that  $N_{\alpha}(\cdot)$ is twice differentiable  and its first and second derivatives satisfy
 
	\begin{align*}
	N_{\alpha} ' (x) [h] & = 2 ( x , h ) + 2 \alpha^2  (\rA^{\frac{1}{2}} x , \rA^{\frac{1}{2}} h )~ x, h \in D(\rA^{\frac{1}{2}}),\\
	N_{\alpha} '' (x) [h, k] & = 2 ( h , k ) + 2 \alpha^2 (\rA^{\frac{1}{2}} h , \rA^{\frac{1}{2}} k)  ~ x,h, k \in D(\rA^{\frac{1}{2}}).
	\end{align*}
	Also, if $x \in D(\rA)$ and $h \in D(\rA^\frac12)$ then 
	\begin{align*}
		N_{\alpha}' (x)[h] = 2 ( (I + \alpha^2 \rA) x, h )= 2 (x,h)+\alpha^2(\rA^\frac12 x, \rA^\frac12 h) =:2 \langle x, h \rangle_{\alpha}.
	\end{align*}

	Now, by applying the  It$\hat{\text{o}}$'s formula to $N_{\alpha}( \mathbf{y}^{\alpha, \delta})$ and then to the function $\varphi (x) = x^p$ and $\Vert\mathbf{y}^{\alpha, \delta}\Vert_{\alpha}^2$ and using the property \eqref{Eq:Cancel-Balpha} we obtain
	\begin{align}\label{high-Integ1}
	&\Vert \mathbf{y}_{h}^{\alpha, \delta} (t) \Vert_{\alpha}^{2p} + p \int_{0}^{t} \Vert \mathbf{y}_{h}^{\alpha, \delta} (s) \Vert_{\alpha}^{2p - 2} \Vert \rA^{\frac{1}{2}} \mathbf{y}_{h}^{\alpha, \delta} \Vert_{\h}^2 ds  + p \alpha^2 \int_{0}^{t} \Vert \mathbf{y}_{h}^{\alpha, \delta} (s) \Vert_{\alpha}^{2p - 2} \Vert \rA \mathbf{y}_{h}^{\alpha, \delta} \Vert_{\h}^2 ds \notag\\
	& \hspace{0.4cm} - p \delta \int_{0}^{t} \Vert \mathbf{y}_{h}^{\alpha, \delta} (s) \Vert_{\alpha}^{2p - 2} \langle \mathbf{y}_{h}^{\alpha, \delta},  \tilde{B}_{\alpha}(\bu,\mathbf{z}_{h}^{\alpha,\delta}) \rangle_\alpha ds \notag\\
	& \hspace{0.4cm}+ 2 \lambda_{\delta}^{-1} (\alpha) \int_{0}^{t} \Vert \mathbf{y}_{h}^{\alpha, \delta} (s) \Vert_{\alpha}^{2p - 2} \langle \tilde{B}_{\alpha}(\bu,J_{\alpha}^{-1}\bu ) - B(\bu, \bu),   \mathbf{y}_{h}^{\alpha, \delta} \rangle_{\alpha} ds \notag \\
	& \leq \Vert \mathbf{y}_{h}^{\alpha, \delta} (0) \Vert_{\alpha}^{2p} +  p  \alpha \lambda_{\delta}^{-2} (\alpha) \sum_{k \geq 1} \int_{0}^{t} \Vert \mathbf{y}_{h}^{\alpha, \delta} (s) \Vert_{\alpha}^{2p - 2} \Vert G(\delta \bu + \lambda_{\delta} (\alpha) \mathbf{y}_{h}^{\alpha, \delta}) e_k\Vert_{ \el^2 }^2 ds \notag\\
	& \hspace{0.4cm} + p (p-1) \alpha \lambda_{\delta}^{-2} (\alpha) \sum_{k \geq 1} \int_{0}^{t} \Vert \mathbf{y}_{h}^{\alpha, \delta} (s) \Vert_{\alpha}^{2p - 4}  \left( G(\delta \bu + \lambda_{\delta} (\alpha) \mathbf{y}_{h}^{\alpha, \delta}) e_k, \mathbf{y}^{\alpha, \delta}\right)^2 ds\notag \\
	& \hspace{0.4cm} +  \alpha^{\frac{1}{2}} \lambda_{\delta}^{-1} (\alpha)  p \int_{0}^{t} \Vert \mathbf{y}_{h}^{\alpha, \delta} (s) \Vert_{\alpha}^{2p - 2} \langle \mathbf{y}_{h}^{\alpha, \delta} ,  \G(\delta \bu + \lambda_{\delta} (\alpha) \mathbf{y}_{h}^{\alpha, \delta}) d \tilde{W}\rangle_{\alpha} . 
	\end{align}
	We need to estimate the terms one by one in this relation.
	For doing this we start with the It\^o correction terms. It is not difficult to see that there exists a constant $C(p)>0$ such that 
	\begin{align*}
	& I_1+ I_2:= p  \alpha \lambda_{\delta}^{-2} (\alpha) \sum_{k \geq 1} \int_{0}^{t} \Vert \mathbf{y}_{h}^{\alpha, \delta} (s) \Vert_{\alpha}^{2p - 2} \Vert G(\delta \bu + \lambda_{\delta} (\alpha) \mathbf{y}_{h}^{\alpha, \delta}) e_k\Vert_{ \el^2 }^2 ds \\
	& + p (p-1) \alpha \lambda_{\delta}^{-2} (\alpha) \sum_{k \geq 1} \int_{0}^{t} \Vert \mathbf{y}_{h}^{\alpha, \delta} (s) \Vert_{\alpha}^{2p - 4}  \left( G(\delta \bu + \lambda_{\delta} (\alpha) \mathbf{y}_{h}^{\alpha, \delta}) e_k, \mathbf{y}^{\alpha, \delta}\right)^2 ds\\
	& \quad \le C(p) \alpha \lambda_{\delta}^{-2} (\alpha) \int_{0}^{t} \Vert \mathbf{y}_{h}^{\alpha, \delta} (s) \Vert_{\alpha}^{2p - 2} \Vert G(\delta \bu + \lambda_{\delta} (\alpha) \mathbf{y}_{h}^{\alpha, \delta})\Vert_{ \mathscr{L}_2(\rK,\h) }^2 ds, 
	\end{align*}
	from which along with the Assumption \ref{Assum-Uniqness} and Remark \ref{Rem:uniqueness} 
	 and the Young inequality we deduce that 
	\begin{align}
	I_1+I_2\le 
	& C(p) \alpha \lambda_{\delta}^{-2} (\alpha)  \int_{0}^{t} \left[1 +  \Vert \mathbf{y}_{h}^{\alpha, \delta} (s) \Vert_{\alpha}\right]^{2p - 2} \left( 1 + \delta^2 \lvert \bu\rvert^2 + \lambda^2_{\delta} (\alpha) \lVert \mathbf{y}_{h}^{\alpha, \delta} \lVert_\alpha^2\right) ds \notag\\
	& \leq C(p) \alpha \lambda_{\delta}^{-2} (\alpha)  \int_{0}^{t} \Vert \mathbf{y}_{h}^{\alpha, \delta} (s) \Vert_{\alpha}^{2p - 2}  \left[1 + \delta^2 | \bu |^2\right] ds + C(p) \alpha \int_0^t \Vert \mathbf{y}_{h}^{\alpha, \delta} (s) \Vert_{\alpha}^{2p} ds \notag\\
	& \leq  C(p) \alpha \lambda_{\delta}^{-2} (\alpha)  \int_{0}^{t} \left[1+\Vert \mathbf{y}_{h}^{\alpha, \delta} (s) \Vert_{\alpha}^{2p}\right]  \left[1 + \delta^2 | \bu |^2\right] ds + C(p) \alpha \int_0^t \Vert \mathbf{y}_{h}^{\alpha, \delta} (s) \Vert_{\alpha}^{2p} ds.\label{Eq:Ito-Correction}
	\end{align}
	Next, 	since $h\in \mathscr{A}_M$, applications of Cauchy-Schwarz's inequality, Assumption \eqref{Assum-Uniqness} and Remark \ref{Rem:uniqueness}, and Young's inequality imply that there exists a constant $C>0$ such that for any $\alpha\in (0,1)$
	\begin{align*}
	\lqq{\int_0^{t} \lVert \ydela(r)\rVert^{2p-2}_\alpha ( G(\delta \bu+\lambda_\delta(\eps)\ydela)h(r),\ydela(r)) dr} \\
	&\le C(1+\lambda_\delta(\alpha))\int_0^{t} \lVert \ydela(r)\rVert^{2p}_\alpha\lVert h(r)\rVert_{\rK} dr  + CM^\frac12 T^\frac12\left(1+ \delta^{2p}\sup_{s\in [0,T]}\lvert \bu(s) \rvert^{2p} \right).
	\end{align*}
 
	The perturbation term containing $\bu$ can be estimated as follows. By applying \eqref{Eq:Est-DiffBandtildeB} we infer that 
	\begin{align*}
	& \lambda_{\delta}^{-1} (\alpha) \int_{0}^{t} \Vert \mathbf{y}_{h}^{\alpha, \delta} (s) \Vert_{\alpha}^{2p - 2} (\tilde{B}_{\alpha}(\bu,J_{\alpha}^{-1}\bu ) - B(\bu, \bu), (I + \alpha^2 A)  \mathbf{y}_{h}^{\alpha, \delta} ) ds \notag \\
	& \le C  \lambda_{\delta}^{-1} (\alpha) \int_{0}^{t} \Vert \mathbf{y}_{h}^{\alpha, \delta}  \Vert_{\alpha}^{2p - 2} [\alpha^2 \lvert \mathbf{y}_{h}^{\alpha, \delta} \rvert + \alpha \lvert \rA^\frac12 \mathbf{y}_{h}^{\alpha, \delta}\rvert+ \alpha^2 \rvert \rA \mathbf{y}_{h}^{\alpha, \delta}\rvert ] \lvert B(\bu,\bu)\rvert ds \\
	&\quad + 
	C  \alpha^2 \lambda_{\delta}^{-1} (\alpha)\int_0^t \Vert \mathbf{y}_{h}^{\alpha, \delta}  \Vert_{\alpha}^{2p - 2}\lvert \rA \mathbf{y}_{h}^{\alpha, \delta} \rvert \lvert \rA \bu\rvert\lvert \rA^\frac12 \bu \rvert ds.
	\end{align*}
	Several applications of  the Young inequality on the right hand side of the last inequality  imply
	\begin{equation}\label{Eq:Intern-Force}
	\begin{split}
	 \lqq{\lambda_{\delta}^{-1} (\alpha) \int_{0}^{t} \Vert \mathbf{y}_{h}^{\alpha, \delta} (s) \Vert_{\alpha}^{2p - 2} (\tilde{B}_{\alpha}(\bu,J_{\alpha}^{-1}\bu ) - B(\bu, \bu), (I + \alpha^2 \rA)  \mathbf{y}_{h}^{\alpha, \delta} ) ds} \\
	   &\le  C(p) \alpha^2 \lambda^{-2}_\delta(\alpha)\int_0^t \left( 1+ \lvert B(\bu,\bu)\rvert^2 + \lvert \rA \bu \rvert^2 \lvert \rA^\frac12 \bu \rvert^2   \right)\left[ 1+ \lVert \mathbf{y}^{\alpha,\delta} \rVert^{2p}_{\alpha}\right] ds  \nonumber \\
	& \, \quad+ \frac{p}{4}\int_0^t \Vert \mathbf{y}_{h}^{\alpha, \delta}  \Vert_{\alpha}^{2p - 2} \left( \lvert \rA^\frac12 \mathbf{y}^{\alpha,\delta} \rvert^2 + \alpha^2 \lvert \rA \mathbf{y}^{\alpha,\delta}\rvert^2 \right) ds. 
	\end{split}
	\end{equation}

	By using the estimate \red{\eqref{2.15-B}}
	 the Young inequality yields and the fact $\lvert y\rvert \le \lVert y \rVert_\alpha$ 
	 
	\begin{eqnarray}
	 	\lqq{p \delta \int_{0}^{t} \Vert \mathbf{y}_{h}^{\alpha, \delta} (s) \Vert_{\alpha}^{2p - 2} \langle \mathbf{y}_{h}^{\alpha, \delta},  \tilde{B}_{\alpha}(\bu,\mathbf{z}_{h}^{\alpha,\delta}) \rangle_\alpha ds} \\
	 	\lqq{= p \delta \int_{0}^{t} \Vert \mathbf{y}_{h}^{\alpha, \delta} (s) \Vert_{\alpha}^{2p - 2} \langle \mathbf{y}_{h}^{\alpha, \delta},  \tilde{B}(\bu,\mathbf{z}_{h}^{\alpha,\delta}) \rangle ds} \\
	&\le & C p  \delta \int_{0}^{t} \Vert \mathbf{y}_{h}^{\alpha, \delta} (s) \Vert_{\alpha}^{2p - 2}\lvert \rA^\frac12 \mathbf{y}_{h}^{\alpha, \delta}\rvert\,  \lvert \rA \bu \rvert \left(\lvert  \mathbf{y}_{h}^{\alpha, \delta}\rvert +\alpha^2 \lvert \rA \mathbf{y}_{h}^{\alpha, \delta}\rvert \right) ds \notag \\
	&\leq&  C(p) \delta^2  \int_{0}^{t} \Vert \mathbf{y}_{h}^{\alpha, \delta} (s) \Vert_{\alpha}^{2p}\lvert \rA \bu(s) \rvert^2 ds \notag \\
	&&+ \frac{p}{4} \int_{0}^{t} \Vert \mathbf{y}_{h}^{\alpha, \delta} (s) \Vert_{\alpha}^{2p - 2} \left(\lvert \rA^\frac12  \mathbf{y}^{\alpha, \delta}\rvert^2+ \alpha^2\vert \rA \mathbf{y}^{\alpha, \delta}\vert^2 \right)ds. 
	\label{Eq:Est-Linear-Pert}
	\end{eqnarray}

	The following estimates are obtained by applying the Burkholder-Davis-Gundy inequality (BDG) and the Young inequality
	\begin{align*}
	&  \mathbb{E} \sup_{0\le s \le t} \left\lvert \int_{0}^{s} \Vert \mathbf{y}^{\alpha, \delta} (s) \Vert_{\alpha}^{2p - 2} \langle  \G (\delta \bu + \lambda_{\delta}(\alpha)\mathbf{y}^{\alpha, \delta}) d \tilde{W}, \mathbf{y}^{\alpha, \delta} \rangle_{\alpha} \right\rvert \\
	& \quad \leq  \mathbb{E} \left[\int_{0}^{t} \Vert \mathbf{y}^{\alpha, \delta} (s) \Vert_{\alpha}^{2p - 2}    \Vert G (\delta \bu + \lambda_{\delta}(\alpha)\mathbf{y}^{\alpha, \delta})\Vert_{\mathscr{L}_2(\rK
		, \h)}^2 \Vert\mathbf{y}^{\alpha, \delta} (s) \Vert_{\alpha}^{2p} ds\right]^{\frac{1}{2}} \notag\\
	&\quad \le \frac12 \mathbb{E} \sup_{0\le s \le t} \lVert \mathbf{y}^{\alpha, \delta}\rVert^{2p}_\alpha+ C \int_{0}^{t} \Vert \mathbf{y}^{\alpha, \delta} (s) \Vert_{\alpha}^{2p - 2}    \Vert G (\delta \bu + \lambda_{\delta}(\alpha)\mathbf{y}^{\alpha, \delta})\Vert_{\mathscr{L}_2(\rK
		, \h)}^2 ds.
	\end{align*}
	Observe that second term of the right hand side of the last inequality can be dealt with the same technique as used in the proof of \eqref{Eq:Ito-Correction}. In particular, we see that 
	\begin{equation}\label{BDG1}
	\begin{split}
	&  \mathbb{E} \sup_{0\le s \le t} \left\lvert \int_{0}^{s} \Vert \mathbf{y}^{\alpha, \delta} (s) \Vert_{\alpha}^{2p - 2} \langle  \G (\delta \bu + \lambda_{\delta}(\alpha)\mathbf{y}^{\alpha, \delta}) d \tilde{W}, \mathbf{y}^{\alpha, \delta} \rangle_{\alpha} \right\rvert \\
	& \quad \le \frac12 \mathbb{E} \sup_{0\le s \le t} \lVert \mathbf{y}^{\alpha, \delta}\rVert^{2p}_\alpha+ C \int_{0}^{t}\left[1+ \Vert \mathbf{y}^{\alpha, \delta} (s) \Vert_{\alpha}^{2p}\right]\left[1+ \delta^2 \lvert \bu \rvert^2+ \lambda_\delta^2(\alpha)\right] ds.
	\end{split}
	\end{equation}
	By plugging the inequalities \eqref{Eq:Ito-Correction}, \eqref{Eq:Intern-Force}, \eqref{Eq:Est-Linear-Pert} and \eqref{BDG1} into \eqref{high-Integ1} and by taking into account the Remark \ref{Rem:Lambdaalpha} we find that there exists a constant $C(p)>0$ such that 
	\begin{align}\label{high-int2}
	& {\mathbb{E}} \sup_{s \in [0, t]}  \Vert \mathbf{y}_{h}^{\alpha, \delta} (t) \Vert_{\alpha}^{2p} + p  {\mathbb{E}} \int_{0}^{t} \Vert \mathbf{y}_{h}^{\alpha, \delta} (t) \Vert_{\alpha}^{2p - 2}\left[ \Vert \rA^{\frac{1}{2}} \mathbf{y}_{h}^{\alpha, \delta} \Vert_{\h}^2 + \alpha^2 \Vert \rA  \mathbf{y}_{h}^{\alpha, \delta} \Vert_{\h}^2\right] ds \notag\\
	& \leq  C(p)  {\mathbb{E}} 
	\int_{0}^{t} \left[1 + \Vert \mathbf{y}_{h}^{\alpha, \delta} (t) \Vert_{\alpha}\right]^{2p}\left( 1+ \lVert h \rVert_{\rK} +\delta^2 [\lvert \rA\bu\rvert^2+\lvert \bu\rvert^2 ] + \lvert B(u,u) \rvert^2+ \lvert \rA \bu \rvert^2\lvert \rA^\frac12 \bu\rvert^2 \right) ds \notag\\
	& \quad +  CM^\frac12 T^\frac12\left(1+ \delta^{2p}\lvert \rA^\frac12 \bu \rvert^{2p} \right) + \Vert \mathbf{y}_{h}^{\alpha, \delta} (0) \Vert_{\alpha}^{2p} 
	\end{align}
	Applying the  Gronwall lemma now yield that there exists a constant $C>0$ such that for any $\alpha \in (0,1)$ we have 
	\begin{align} \label{Estim-fin}
	&  {\mathbb{E}} \sup_{t \in [0, T]} \Vert \mathbf{y}_{h}^{\alpha, \delta} (t) \Vert_{\alpha}^{2p} +    {\mathbb{E}} \int_{0}^{t} \Vert \mathbf{y}_{h}^{\alpha, \delta} (t) \Vert_{\alpha}^{2p - 2} \Vert\mathbf{y}_{h}^{\alpha, \delta} \Vert_{\ve}^2 ds   \notag \\
	& \leq C\left( 1+  \lvert \rA^\frac12  \mathbf{y}_{h}^{\alpha, \delta} (0)\rvert^2 + CM^\frac12 T^\frac12\left(1+ \delta^{2p}\lvert \rA^\frac12 \bu \rvert^{2p} \right)\right)e^{\int_0^T\Phi_\delta (s) ds},
	\end{align}
	where 
	\begin{equation*}
	\Phi:= 1+ \lVert h \rVert_{\rK}+\delta^2 [\lvert \rA\bu\rvert^2+\lvert \bu\rvert^2 ] + \lvert B(u,u) \rvert^2+ \lvert \rA \bu \rvert^2\lvert \rA^\frac12 \bu\rvert^2.
	\end{equation*}
	This completes the proof of Theorem \ref{THM:SCNS-alpha-delta1}.
\end{proof}
 
\subsection{Analysis of the deterministic controlled Navier-Stokes-$\alpha$}
In this subsection we fix $h \in \el^2(0,T; \mathrm{K})$ and  analyse the following deterministic controlled Navier-Stokes-$\alpha$ model:
\begin{subequations}\label{DCNS-alpha-delta}
	\begin{align}
	& d \mathbf{y}^{\delta} + \rA \mathbf{y}^{\delta} dt + (1- \delta) B(\mathbf{y}^{\delta}, \mathbf{y}^{\delta}) dt + \delta B (\bu, \mathbf{y}^{\delta}) + \delta B (\mathbf{y}^{\delta}, \bu) = G(\delta \bu + (1 - \delta) \mathbf{y}^{\delta}) h, \label{Eq:DCNS-alpha-delta1}    \\ 	 
	& \mathbf{y}^{\delta} (0) = (1 - \delta) \xi . \label{Eq:DCNS-alpha-delta} 
	\end{align}
\end{subequations}
 
The main result of this subsection is given in the following theorem.
\begin{Thm} \label{THM:DCNS-alpha-delta}
Let $h \in \el^2(0,T; \mathrm{K})$ and $\xi \in \ve$. Then, \eqref{DCNS-alpha-delta} has a unique solution $\mathbf{y}_h^{\delta} \in C([0, T]; \h^1) \cap \el^2 (0, T; \h^2)$. Moreover, if $h \in \mathcal{A}_M,\, M>0,$ then there exists a deterministic constant $C > 0$, which depends only on $M$ and $\lVert \xi \rVert$, such that with probability $1$
\begin{align}\label{Eq:Est-Det-Control}
\sup_{t \in [0, T]}( \lvert \mathbf{y}_{h}^{\delta} (t)  \rvert^2+ \Vert \mathbf{y}_{h}^{\delta} (t) \Vert^2 )+ \int_{0}^{T} \left(\rvert \rA \mathbf{y}_{h}^{\delta} (t)  \rvert^2+ \rvert \rA \mathbf{y}_h^{\delta} (t) \rvert^2 \right)dt \leq C.
\end{align}
\end{Thm}
\begin{Rem}
	Note that when $\delta=0$ and $h=0$ the above theorem provides also the following estimates for $\bu$, the unique solution to \eqref{Eq: ABS-NSE}: 
	\begin{align}\label{Eq:Est-Det-NSE}
	\sup_{t \in [0, T]} \Vert \bu (t) \Vert^2 + \int_{0}^{T} \vert \rA  \bu (t) \vert^2 dt \leq C.
	\end{align}
	This estimate could be found in many classical literature for the Navier-Stokes equations such as \cite{TEMAMA} and \cite{PC+CF-BlueBook}.
\end{Rem}
\begin{proof}[Proof of Theorem \ref{THM:DCNS-alpha-delta}]
Since the system  \eqref{DCNS-alpha-delta} is the Navier-Stokes with the linear perturbations $ \delta B (\bu, \mathbf{y}^{\delta}) + \delta B (\mathbf{y}^{\delta}, \bu)$ and the Lipschitz continuous perturbations $G(\delta \bu + (1 - \delta) \mathbf{y}^{\delta}) h,$  we can prove the existence and uniqueness results in the above theorem by following the standard scheme of proof for the Navier-Stokes equations, see, for instance, \cite{TEMAMA}. Since this is now standard  we only focus on deriving the crucial estimates for the solutions.  
 {For the sake of simplicity, we will just write $ \mathbf{y}^\delta $ instead of  $\mathbf{y}^\delta_h$}. We will also suppress the dependence of  $ \mathbf{y}^\delta $ on the time variable. 
 
By formally multiplying the first equation in \eqref{DCNS-alpha-delta} by $(I+\rA)\mathbf{y}^{\delta} (t) $ we find that 
\begin{align}\label{Eq:deterministic-0delta1}
& \frac{1}{2} \frac{d}{d t}\left( \lvert \mathbf{y}^\delta\rvert^2 + \lvert \rA^{\frac{1}{2}} \mathbf{y}^\delta \rvert^2 \right) +  \rvert \rA  \mathbf{y}^\delta \rvert^2 + \rvert \rA^{\frac{1}{2}} \mathbf{y}^\delta \rvert^2 + (1 - \delta) ( (I + \rA)  \mathbf{y}^\delta , B(\mathbf{y}^\delta, \mathbf{y}^\delta)) \notag \\
& \hspace{0.4cm} + \delta ( (I + \rA)  \mathbf{y}^\delta  , B(\bu, \mathbf{y}^\delta) + B(\mathbf{y}^\delta, \bu) ) = \left( (I + \rA)  \mathbf{y}^\delta, G(\delta \bu + (1 - \delta) \mathbf{y}^{\delta}) h (t) \right).
\end{align}
Observe that since we are working on the torus
\begin{align*}
	 (\mathbf{y}^\delta(I+\rA),B(\mathbf{y}^\delta,\mathbf{y}^\delta)) = 0 \quad \text{and } \quad ( B(\bu,\mathbf{y}^\delta),\mathbf{y}^\delta)= 0.
\end{align*}
By the H\"older, the Gagliardo-Nirenberg and the Young inequalities we obtain
\begin{align*}
	\lqq{\delta ( B(\bu,\mathbf{y}^\delta) + B(\mathbf{y}^\delta, \bu), (I+\rA)\mathbf{y}^\delta)} \\
	&\leq \delta |C(I+\rA)\mathbf{y}^\delta|[|\bu|_{\el^\infty}|\rA^\frac12 \mathbf{y}^\delta| + |\mathbf{y}^\delta|_{\el^2}|\nabla \bu|_{\el^2}] \\
	&\leq \delta C[|\mathbf{y}|^\delta + |\rA\mathbf{y}^\delta|] |\rA\bu||\rA^\frac12 \mathbf{y}^\delta| \\
	&\leq \frac14 |\rA\mathbf{y}^\delta|^2 + C\delta[1+|\rA \bu|^2][|\mathbf{y}^\delta|^2 + |\rA^\frac12 \mathbf{y}^\delta|].
\end{align*}

We will now deal with the term containing $G.$ By using the Cauchy-Schwartz, the Young inequalities and the Assumption \ref{Assum-Uniqness} we see that 

\begin{align*}
	\lqq{ (G(\delta \bu + (1-\delta)\mathbf{y}^\delta)h , (I+\rA)\mathbf{y}^\delta    )}\\ 
	&\leq G |(I+\rA)\mathbf{y}^\delta| |h|_\rK |G(\delta \bu + (1-\delta)\mathbf{y}^\delta)|_{\mathscr{L}_2(\rK,\h)} \\
	&\leq \frac14 |\rA\mathbf{y}^\delta|^2 + C|\mathbf{y}^\delta|^2 [ |h|_\rK^2+(1-\delta)^2] + C|h|_\rK^2[1+\delta^2 |\bu|^2].
\end{align*}

Collecting these inequalities together implies

\begin{align*}
	\lqq{\frac12 \frac{d}{dt}[|\mathbf{y}^\delta|^2 + |\rA^\frac12\mathbf{y}^\delta|] + |\rA^\frac12\mathbf{y}^\delta|^2 + \frac12|\rA\mathbf{y}^\delta|^2 }\\
	&\leq C[|\mathbf{y}^\delta|^2 + |\rA^\frac12\mathbf{y}^\delta|][\delta(1+|\rA\bu|^2) + |h|_\rK^2 + (1-\delta)^2] + C|h|_\rK^2 [1+\delta^2 |\bu|^2].
\end{align*}
Applying the Gronwall's inequality yields that there exists a constant $C>0$, depending only on $M$ and $\lvert \rA \xi \rvert$, such that 
\begin{align}
	\sup_{t\in [0,T]} (|\mathbf{y}^\delta(t)|^2 + |\rA^\frac12 \mathbf{y}^\delta(t)|^2) + \int_0^T[|\rA^\frac12 \mathbf{y}^\delta(t)|^2 + \lvert \rA \mathbf{y}^\delta(t)]dt\le C, 
\end{align}
which completes the proof of Theorem  \ref{THM:DCNS-alpha-delta}.
\end{proof}

The above theorem enables us to define a map $\Gamma^{0,\delta}_\xi: C([0,T];\mathrm{K}) \to C([0,T];\h)\cap \el^2(0,T;D(\rA^{\frac 12}))$  by setting
\begin{itemize}
	\item $\Gamma_\xi^{0,\delta}(x)$ is the unique solution $\bu_h^\delta$ to \eqref{Eq:DCNS-alpha-delta1}  if $x=\int_0^\cdot h(r) dr$, $h \in \el^2(0,T;\mathrm{K})$;
	\item  $\Gamma_\xi^{0,\delta}(x)=0$ otherwise.
\end{itemize} 
We will see in the next theorem and  Remark \ref{Rem-Borel-Meas} that this map is in fact Borel measurable.

{We now state and prove the following two important results.
}
	\begin{Thm} \label{THM:LDP1}
		Let $\delta$ and $\xi \in D(\rA^{\frac{1}{2}})$. Then, the set 
		$\{ \Gamma^{0, \delta}_{\xi} (\int_0^\cdot h(s)ds): h \in \mathcal{\rA}_M\}$ is a compact set of  $C([0, T]: \h)\cap \el^2(0,T;D(\rA^\frac12))$. \\ 
 
	\end{Thm}
\begin{Rem}\label{Rem-Borel-Meas}
	The above proposition amounts to say that if $(h_n)_{n\in \mathbb{N}}\subset \mathcal{A}_M$, $M>0$, is a sequence that converges weakly to $h \in \mathcal{A}_M$, then $\Gamma^{0,\delta}_\xi \left(\int_0^{\cdot}h_n(r)dr \right)$ strongly converges to $\Gamma_\xi^{0,\delta(\int_0^{\cdot} h(r) dr)}$ in $C([0,T];\h)\cap \el^2(0,T;D(\rA^{\frac 12}))$. Consequently, the map 
	$$
	\mathcal{A}_M \ni h \mapsto \Gamma^{0,\delta}_\xi(\int_0^{\cdot} h(r)dr)\in C([0,T];\h)\cap \el^2(0,T;D(\rA^\frac12)),
	$$ 
	is Borel measurable. 
\end{Rem}
\begin{proof}[Proof of Theorem \ref{THM:LDP1}]
	Let $(h_n)_{n\in \mathbb{N}}\subset \mathcal{A}_M$ and $h\in \mathcal{A}_M$ such that  
	\begin{align*}
		h_n\to h \qquad \text{weakly in} \qquad \el^2(0,T;\rK).
	\end{align*}
Let us denote by $\mathbf{y}_n= \Gamma^{0, \delta}(\int_0^\cdot h(s)ds),\, n\in\mathbb{N}$. Then by {Theorem \ref{THM:DCNS-alpha-delta}}  there exists a constant $C>0$ such that for all $n\in\mathbb{N}$
\begin{align}
	\sup_{t\in [0,T]}\left(|\mathbf{y}_n(t)|^2 +  |\rA^\frac12 \mathbf{y}_n(t)|^2  \right) + \int_0^T\left(|\rA^\frac12 \mathbf{y}_n |^2  + |\rA  \mathbf{y}_n(t)|^2\right)ds
	< C.  \label{Eq:UNIF-YN} 
\end{align}
Furthermore, 
\begin{align*}
	|\partial_t\mathbf{y}_n|\leq |\rA\mathbf{y}_n| + (1-\delta)|B(\mathbf{y}_n,\mathbf{y}_n)| + \delta[|B(\bu,\mathbf{y}_n)| + |B(\mathbf{y}_n, \bu)|]
\end{align*}
Now, observe that by making use of the H\"older, the Gagliardo-Nirenberg inequalities we obtain that there exists $C>0$ such that for all $n\in\mathbb{N}$
\begin{align*}
\lqq{(1-\delta)|\mathbf{y}_n|_{\el^\infty}|\rA^\frac12\mathbf{y}_n| + \delta[|\bu|_{\el^\infty}|\rA^\frac12 \mathbf{y}_n| + |\mathbf{y}_n|_{\el^\infty}|\rA^\frac12 \bu|] }\\
&\leq C|\rA\mathbf{y}_n|[(1-\delta)|\rA^\frac12\mathbf{y}_n| + \delta|\rA^\frac12 \bu|] + C\delta|\rA^\frac12 \mathbf{y}_n||\rA\bu|.
\end{align*}
Hence,
\begin{align}
	 \int_0^T|\partial_t\mathbf{y}_n|^2 dt\leq& C[(1-\delta)\sup_t|\rA^\frac12\mathbf{y}_n|^2 + \delta \sup_t|\rA^\frac12 \bu|^2 ] \int_0^T|\rA\mathbf{y}_n|^2ds + C\delta \sup_t|\rA^\frac12\mathbf{y}_n|^2\int_0^T|\rA\bu|^2ds \nonumber  \\
	\leq & C[(1-\delta) + \delta\sup_t|\rA^\frac12 \bu|^2]  + C\delta \int_0^T|\rA \bu|^2ds \label{Eq:UNIF-EST-dy}
\end{align}
The estimates \eqref{Eq:UNIF-YN} and  \eqref{Eq:UNIF-EST-dy} imply that 
\begin{itemize}
	\item $(\mathbf{y}_n)_n$ is uniformly bounded in $C([0,T];D(\rA^\frac12))\cap \el^2(0,T;D(\rA))$.
	\item $ (\partial_t\mathbf{y}_n)_n $ is uniformly bounded in $\el^2(0,T;\h)$.
\end{itemize}
Hence, by the Banach-Alaoglu and the Aubin-Lions theorem there exist a subsequence, still denoted by $\mathbf{y}_n$, of $\mathbf{y}_n$ and $\mathbf{y}$ such that 
\begin{align}
	&\mathbf{y}_n \to \tilde{\mathbf{y}} \quad \text{ weak-* in } \el^\infty(0,T;D(\rA^\frac12)) \nonumber \\
	&\mathbf{y}_n \to \tilde{\mathbf{y}} \quad \text{ weak in } \el^2(0,T;D(\rA )) \label{Eq:WEAK}\\
	&\mathbf{y}_n \to \tilde{\mathbf{y}} \quad \text{ strong in } \el^2(0,T;D(\rA^\frac12 )) \label{Eq:STRONG} \\
	&\partial_t\mathbf{y}_n \to \partial_t\tilde{\mathbf{y}} \quad \text{ weak in } \el^2(0,T;\h) \nonumber
\end{align}
The last convergence and the first one imply that 
\begin{align*}
	\tilde{\mathbf{y}}\in C([0,T];\h)\cap C_w([0,T];D(\rA^\frac12)),
\end{align*}
where $C_w([0,T];D(\rA^\frac12))$ denotes the space of functions $f:[0,T]\to X,$ $X$ a given Banach space, that are weakly continuous.
By passing to the limits we shall show that $\tilde{\mathbf{y}}$ is a solution to the system \eqref{DCNS-alpha-delta}.
In fact by arguing exactly as in \cite{TEMAMA} we see that 
\begin{align*}
	B(\mathbf{y}^n,\mathbf{y}^n) \to B(\tilde{\mathbf{y}},\tilde{\mathbf{y}}) \quad \text{ in} \quad \el^2(0,T;\h).
\end{align*}
Since $B(\bu,\mathbf{y}^n)$, $B(\mathbf{y}^n, \bu)$ and $\rA\mathbf{y}^n$ are linear continuous $D(\rA^\frac12),\, D(\rA^\frac12)$ and $ D(\rA) $, respectively, by using the strong convergence \eqref{Eq:STRONG} and the weak convergence \eqref{Eq:WEAK} we obtain 
\begin{itemize}
	\item $B(\bu,\mathbf{y}_n)+B(\mathbf{y}_n, \bu) \to B(\bu,\tilde{\mathbf{y}})+B(\tilde{\mathbf{y}},\bu)$ strong in $\el^2(0,T;\h)$.
	\item $\rA\mathbf{y}_n \to \rA\tilde{\mathbf{y}}$ weak in $\el^2(0,T;\h).$
\end{itemize}
Note that the first convergence holds because $\bu\in C([0,T],D(\rA^\frac12))\cap \el^2(0,T;D(\rA)).$

What remains to prove is that 
\begin{align}
	G(\delta \bu + \lambda_\delta(\alpha)\mathbf{y}_n)h_n \to G(\delta \bu + \lambda_\delta(\alpha)y)h \quad\text{ in  } \el^2(0,T;\h). \label{Eq:CONV-CONTROL}
\end{align}
To prove this fact we first observe that  the Assumption \ref{Assum-Uniqness} and the strong convergence \eqref{Eq:STRONG} imply

\begin{align*}
	G(\delta \bu + \lambda_\delta(\alpha)\mathbf{y}_n) \to G(\delta \bu + \lambda_\delta(\alpha)\tilde{\mathbf{y}}) \quad \text{ strongly in }\quad \el^2(0,T;\mathscr{L}_2(\rK,\h)).
\end{align*}
 This along with the assumption that $h_n\to h$ weak in $\el^2(0,T;\rK)$ implies that convergence \eqref{Eq:CONV-CONTROL} holds true.
 
 By collecting all the above convergences, it is not difficult to see that $\tilde{\mathbf{y}}$ is a solution to \eqref{DCNS-alpha-delta}.
 By uniqueness, $\tilde{\mathbf{y}} = \mathbf{y}^\delta_h = \Gamma^{0,\delta}(\int_0^\cdot h(s)ds)$and the whole sequence $\mathbf{y}_n$ converges to $\mathbf{y}^\delta_h$ . This completes the proof of Theorem \ref{THM:LDP1}.
\end{proof}
\section{The deviation principles result and its proofs}
\label{sec:5Deviation}
\subsection{Formulation of the main results}
This section is the heart of this paper. We will state and prove our main results in this section, but before doing so we briefly recall few definitions  from the LDP theory. 

Let $\mathcal{E}$ be a Polish space and $\mathcal{B}(\mathcal{E})$ its Borel $\sigma$-algebra.
\begin{Def}
	A function $I: \mathcal{E} \to [0,\infty]$ is a (good) rate function if it is lower semicontinuous and the level sets $\{e\in\mathcal{E}; I(e)\le a   \}$, $a\in [0,\infty)$, are compact subsets of $\mathcal{E}$.
\end{Def}
Next let $\varrho$ be a real-valued map defined on $[0,\infty)$ such that 
$$ \varrho(\eps) \to  \infty \text{ as } \eps \to \infty.$$
\begin{Def}
	Let $(\Omega, \mathscr{F},\mathbb{P})$ be a complete probability space. An $\mathcal{E}$-valued random variable $(X_\eps)_{\eps \in (0,1]}$ satisfies the LDP on $\mathcal{E}$ with speed $\varrho(\eps)$ and rate function $I$ if and only if the following two conditions hold
	\begin{enumerate}[label = (\alph{*})]
		\item for any closed set $F\subset \mathcal{E}$
		$$ \limsup_{\eps \to 0} \varrho^{-1}(\eps) \log \mathbb{P}(X_\eps \in F) \le - \inf_{x\in F} I(x);$$
		\item for any open set $O\subset \mathcal{E}$
		$$ \liminf_{\eps \to 0}\varrho^{-1}(\eps) \log \mathbb{P}(X_\eps \in O)\ge -\inf_{x\in O} I(x). $$
	\end{enumerate} 
\end{Def}
We are now ready to state our main results. 
\begin{Thm}\label{Main-Theorem-Delta}
	Let $\delta \in \{0,1\}$, $\xi \in D(\rA^{\frac{1}{2}})$ and Assumption \ref{Assum-Uniqness} hold. Then, the family $(\bu^{\alpha,\delta})_{\alpha \in (0,1]}$ satisfies an LDP on $C([0,T];\h)\cap \el^2(0,T;D(\rA^{\frac{1}2 }))$ with speed $\alpha^{-1}\lambda_\delta^{2}(\alpha)$ and rate function $I_\delta$ given by 
	\begin{equation*}
	I_\delta(x )= \inf_{\{h \in \el^2(0,T; \mathrm{K}): x = \Gamma^{0,\delta}_\xi (\int_0^{\cdot} h(r) dr) \}  } \biggl\{\frac12 \int_0^T \lVert h(r)\rVert^2_{\mathrm{K} } dr \biggr\}.
	\end{equation*}
	As usual, we understand that $\inf \emptyset = \infty.$ 
\end{Thm}
\begin{proof}
	The proof requires few preparations and hence it will be postponed to  Subsection \ref{Sec:Proof-of-Main-Thm}.
\end{proof}
We can divide the result in the above theorem into two parts which will form the following two corollaries. They give the LDP and MDP on $C([0,T];\h)\cap \el^2(0,T; D(\rA^{\frac{1}2 }))$ for the solution $\bu^\alpha$ to \eqref{Eq: ABS-LANS-ALPHA}. 
\begin{Cc}
	Let $s\in \{0,1\}$, $\xi \in D(\rA^{\frac{1}2 })$ and  $G$ satisfies Assumption \ref{Assum-Uniqness}. Then, the family of solutions $(\bu^\alpha)_{\alpha \in (0,1]}$ to  \eqref{Eq: ABS-LANS-ALPHA} satisfies an LDP on $C([0,T];D(\rA^{1+\frac{s}2} ))$ with speed $\alpha^{-1}$ and  rate function $I_0$ given by 
	\begin{equation}
	I_0(x)= \inf_{\{h \in \el^2(0,T; \mathrm{K}): x= \Gamma^{0,0}_\xi (\int_0^{\cdot} h(r) dr) \}  } \biggl\{\frac12 \int_0^T \lVert h(r)\rVert^2_{\mathrm{K} } dr \biggr\}.
	\end{equation} 
\end{Cc}
\begin{Cc}
	If $\xi \in D(\rA^{\frac{1}2 })$ and  $G$ satisfies Assumption \ref{Assum-Uniqness}, then $\left(\alpha^{-\frac12}\lambda^{-1}(\alpha)[\bu^\alpha-\bu]\right)_{\alpha \in (0,1]}$ satisfies an LDP on $C([0,T];\h)\cap \el^2(0,T; D(\rA^{\frac{1}2}))$ with speed $\lambda^2(\alpha)$ and rate function $I_1$ given by 
	\begin{equation}
	I_1(x)= \inf_{\{h\in \el^2(0,T; \mathrm{K}): x= \Gamma^{0,1}_\xi (\int_0^{\cdot} h(r) dr) \}  } \biggl\{\frac12 \int_0^T \lVert h(r)\rVert^2_{\mathrm{K} } dr \biggr\}.
	\end{equation} 
\end{Cc}

\subsection{Intermediate results}
In order to prove Theorem \ref{Main-Theorem-Delta} we will use the weak convergence approach to LDP and Budhiraja-Dupuis' results on representation of functionals of Brownian motion, see \cite{BuDu} and \cite{AB+PD+VM}. These require few intermediate results which are stated and proved below.
\begin{Lem}\label{Prop:Conv-In-Proba}\label{Lem-Prop-Imp-1}
	Let $M>0,\,(h_n)_n\subset \mathcal{A}_M$ and $(\alpha_n)_{n \in \mathbb{N}}$ be sequence such that $\alpha_n \to 0$ as $n \to \infty$.  Let $\mathbf{y}_n = \Gamma^{\alpha_n,\delta}\left(W + \alpha_n^{-\frac12}\lambda_\delta(\alpha_n)\int_0^\cdot h_n(s)ds \right)$ and $\mathbf{z}_n = \Gamma^{0,\delta}\left(\int_0^\cdot h_n(s)ds \right)$. 
	Then, for any $\eps>0$
	\begin{align*}
		\lim_{n\to\infty}\mathbb{P}\left(\left[\sup_{t\in [0,T]}|\mathbf{y}_n(t) - \mathbf{z}_n(t)|^2 + \int_0^{T} |\rA^\frac12(\mathbf{y}_n-\mathbf{z}_n)|^2 ds\right] >\eps\right) = 0 .
	\end{align*}
\end{Lem}
Before proving lemma we state the following important remark.
\begin{Rem}\label{Rem-ConvProba}
	Observe that if  $h_n\equiv 0, \; \forall n \in \mathbb{N},$ and $\delta=0$, then the above lemma gives a result on the convergence in probability of the solutions of the stochastic LANS-$\alpha$ \eqref{Eq: ABS-LANS-ALPHA} to the solution of the stochastic NSE \eqref{Eq: ABS-NSE} as $\alpha \to 0$. In fact, 
	 $\mathbf{y}_n = \Gamma^{\alpha_n,0}\left(W  \right)$ and $ \Gamma^{0,0}\left(0\right)=\bu$ are the unique solutions to the stochastic LANS-$\alpha$ \eqref{Eq: ABS-LANS-ALPHA} and to the stochastic NSE \eqref{Eq: ABS-NSE}, respectively.
\end{Rem}
\begin{proof}[Proof of Lemma \ref{Lem-Prop-Imp-1}]
Let $\mathbf{y}_n$ and $\mathbf{z}_n$ be as in the statement of the theorem.  
Let us put $\mathbf{w}_n = \mathbf{y}_n-\mathbf{z}_n$. Let $\tau_{n, N}$ be the stopping time defined  
$$\tau_{n, N} =\inf\{t\in[0,T]:|\mathbf{y}_n(t)|\geq N\}\wedge T,\, N\geq 0.$$ 
For the sake of  simplification we just write $\alpha$  instead of $\alpha_n$ throughout this proof. Also, we simply write $\tau_N$ in place of $\tau_{n,N}$.

{Since $\mathbf{y}_n$ and $\mathbf{z}_n$ are the unique solutions to the stochastic controlled and deterministic controlled systems, respectively, }
it is not difficult to see that $\mathbf{w}_n$ satisfies 
\begin{eqnarray*}
	\lqq{d\mathbf{w}_n + A\mathbf{w}_n + \lambda_\delta(\alpha)\tilde{B}_\alpha(\mathbf{y}_n,J^{-1}_\alpha\textbf{y}_n) - (1-\delta)B(\mathbf{z}_n,\mathbf{z}_n) } \\
	&&+ \delta [\tilde{B}_\alpha(\bu,J^{-1}_\alpha\mathbf{y}_n)-B(\bu,\mathbf{z}_n) + \tilde{B}_\alpha(\bu,J^{-1}_\alpha \bu)-B(\mathbf{z}_n,\bu)] \\
	&=&\delta\lambda^{-1}_\delta(\alpha)[B(\bu,\bu)-\tilde{B}_\alpha(\bu,J^{-1}_\alpha  \bu)] \\
	 &&+ \G (\Psi_n)h_n - G(\Phi_n)h + \alpha^\frac12 \lambda_\delta^{-1}(\alpha) \G (\Psi_n)dW,
\end{eqnarray*}
where $\Psi_n = \delta \bu + \lambda_\delta(\alpha)\mathbf{y}_n$ and $\Phi_n = \delta \bu + (1-\delta)\mathbf{z}_n.$
 
Let 
\begin{align*}
\mathbf{N}[\mathbf{y}_n ,  \mathbf{z}_n ] & = \lambda_{\delta} (\alpha) \tilde{B}_\alpha (\mathbf{y}_n, J_\alpha^{-1}\mathbf{y}_n) - (1 - \delta) B(\mathbf{z}_n, \mathbf{z}_n) ~~ \text{ and } \\
\mathbf{L}[\mathbf{y}_n, \mathbf{z}_n] & = \delta \left[ \tilde{B}_\alpha (\bu,  J_\alpha^{-1} \mathbf{y}_n) - B(\bu, \mathbf{z}_n) + \tilde{B}_\alpha (\mathbf{y}_n, J_\alpha^{-1} \bu) - B(\mathbf{z}_n , \bu)\right].
\end{align*}
By applying the It$\hat{\text{o}}$'s formula to $\varphi(x) = \Vert x\Vert_\alpha = |x|^2 + \alpha^2 \left\lvert \rA^{\frac{1}{2}}  x\right\rvert^2$ and to the process $\mathbf{w}_n$,  taking the supremum and the mathematical expectation to the  resulting  equation we obtain
\begin{align}
&\Vert \mathbf{w}_n (t \wedge \tau_N)\Vert_\alpha^2 + 2\int_{0}^{t \wedge \tau_N} \left[ \left\lvert \rA^{\frac{1}{2}}  \mathbf{w}_n (s)\right\rvert^2  + \alpha^2 \left\lvert \rA   \mathbf{w}_n (s)\right\rvert^2 \right] ds \notag\\
&\le  2 \int_{0}^{t \wedge \tau_N}\lvert  (\mathbf{N} [\mathbf{y}_n, \mathbf{z}_n] + \mathbf{L}[\mathbf{y}_n, \mathbf{z}_n], J_\alpha^{-1} \mathbf{w}_n)\rvert ds \notag \\
& + 2 \delta \lambda_{\delta}^{-1}(\alpha)  \int_{0}^{t \wedge \tau_N}\lvert ( B(\bu, \bu) - \tilde{B}_\alpha (\bu , J_\alpha^{-1} \bu) , J_\alpha^{-1} \mathbf{w}_n) \rvert ds\\
& + 2 \int_{0}^{t \wedge \tau_N} \lvert (\G  (\Psi_n ) h_n ,\mathbf{G}(\Phi) h_n , J_\alpha^{-1} \mathbf{w}_n )\rvert  ds \notag \\
& + \alpha \lambda_{\delta}^{-2} (\alpha) \mathbb{E} \int_{0}^{t \wedge \tau_N} \Vert G(\Psi_n)\Vert_{\mathscr{L} (\rK, \h)}^2 ds + \alpha^\frac12 \lambda_{\delta}^{-1} (\alpha) \int_{0}^{t \wedge \tau_N} (\mathbf{w}_n , G (\Psi_n) dW ). \label{Eq:Conv-Prob-Start}
\end{align}
Using \eqref{Eq:Est-DiffBandtildeB} and the Young's inequality we see that
\begin{align}
& 2 \delta \lambda_{\delta}^{-1}(\alpha)  \int_{0}^{t \wedge \tau_N} ( B(\bu, \bu) - \tilde{B}_\alpha (\bu , J_\alpha^{-1} \bu) , J_\alpha^{-1} \mathbf{w}_n) ds \notag\\
& \leq 2 \delta \lambda_{\delta}^{-1}(\alpha) \int_{0}^{t \wedge \tau_N}\left( \left[ \dfrac{\alpha}{2} \left\lvert \rA^{\frac{1}{2}}\mathbf{w}_n\right\rvert + \alpha^2 |\rA \mathbf{w}_n|^2  \right]|B(\bu, \bu)| + C \alpha^2 \left\lvert \rA^{\frac{1}{2}} \mathbf{w}_n\right\rvert^2 |\rA \bu|^2\right) ds \notag\\
& \leq \int_{0}^{t \wedge \tau_N}\left(\dfrac{1}{24} \left\lvert \rA^{\frac{1}{2}} \mathbf{w}_n \right\rvert^2 + \dfrac{\alpha^2}{24} |\rA \mathbf{w}_n|^2 + C \left[\alpha^2 \lambda_{\delta}^{-1}(\alpha) \left\lvert \rA^{\frac{1}{2}} \bu \right\rvert^2 + \alpha^2 \right] |\rA\bu|^2 \right) ds. \label{Eq:Est-ProbConv-0}
\end{align}
It follows from the bilinearity of $B$ and $\tilde{B}$, and the equations \eqref{Eq:Cancel-Prop-B}, \eqref{Eq:Cancel-Prop-tildeB} that 
\begin{equation}
\begin{split}
&\lambda_{\delta}(\alpha) \left(\tilde{B}_\alpha (\mathbf{y}_n, J_\alpha^{-1} \mathbf{y}_n) - \tilde{B}_\alpha (\mathbf{z}_n, J_\alpha^{-1} \mathbf{z}_n) , J_\alpha^{-1} \mathbf{w}_n \right)\\
& \quad = \lambda_{\delta}(\alpha) \left( \tilde{B} (\mathbf{w}_n, \mathbf{w}_n),  \mathbf{z}_n \right) + \alpha^2 \lambda_{\delta} (\alpha) \left( \tilde{B} (\mathbf{z}_n, A \mathbf{w}_n) ,
\mathbf{w}_n\right). 
\end{split} 
\end{equation}

Thanks to the estimate \eqref{refine-2.10} and the Young inequality we obtain
\begin{align}
\lambda_{\delta}(\alpha) \left( B(\mathbf{w}_n, \mathbf{w}_n) , \mathbf{z}_n\right) 
& = - \lambda_{\delta}(\alpha) \left( B(\mathbf{w}_n, \mathbf{z}_n) , \mathbf{w}_n\right) \notag\\
& \leq \lambda_{\delta}(\alpha) |\mathbf{w}_n| \left\lvert \rA^{\frac{1}{2}} \mathbf{z}_n \right\rvert^{\frac{1}{2}} |\rA\mathbf{z}_n|^{\frac{1}{2}} |\mathbf{w}_n|^{\frac{1}{2}} \left\lvert \rA^{\frac{1}{2}} \mathbf{w}_n \right\rvert^{\frac{1}{2}} \notag \\
& \leq \lambda_{\delta}(\alpha) |\mathbf{w}_n|^{\frac{3}{2}} \left\lvert \rA^{\frac{1}{2}} \mathbf{z}_n \right\rvert^{\frac{1}{2}} |\rA\mathbf{z}_n|^{\frac{1}{2}}   \left\lvert \rA^{\frac{1}{2}} \mathbf{w}_n \right\rvert^{\frac{1}{2}} \notag \\
& \leq \dfrac{1}{24} \left\lvert \rA^{\frac{1}{2}} \mathbf{w}_n \right\rvert^2 + \lambda_{\delta}^{\frac{4}{3}} (\alpha) |\mathbf{w}_n|^2 \left\lvert \rA^{\frac{1}{2}} \mathbf{z}_n \right\rvert^{\frac{2}{3}} |\rA\mathbf{z}_n|^{\frac{2}{3}}.
\end{align}
We now proceed in estimating the term 
$
\alpha^2 \lambda_{\delta}(\alpha)  (\tilde{B}(\mathbf{z}_n, \rA\mathbf{w}_n), \mathbf{w}_n ).
$
For doing so we utilize \eqref{2.15-B} and the Young's inequality and find that 
\begin{align}
\alpha^2 \lambda_{\delta}(\alpha)  (\tilde{B}(\mathbf{z}_n, \rA\mathbf{w}_n), \mathbf{w}_n )
& \leq \alpha^2 \lambda_{\delta}(\alpha)  |\rA \mathbf{w}_n| |\rA\mathbf{z}_n| \left\lvert \rA^{\frac{1}{2}} \mathbf{w}_n \right\rvert \notag\\
& \leq \dfrac{\alpha^2}{24} |\rA \mathbf{w}_n|^2 + C \lambda_{\delta}^2(\alpha) \alpha^2 \left\lvert \rA^{\frac{1}{2}} \mathbf{w}_n \right\rvert^2 |\rA\mathbf{z}_n|^2.
\end{align}
Thus,
\begin{align}
&\lambda_{\delta}(\alpha) \left(\tilde{B}_\alpha (\mathbf{y}_n, J_\alpha^{-1} \mathbf{y}_n) - \tilde{B}_\alpha (\mathbf{z}_n, J_\alpha^{-1} \mathbf{z}_n) , J_\alpha^{-1} \mathbf{w}_n \right) \notag \\
& \leq \dfrac{1}{24} \left\lvert \rA^{\frac{1}{2}} \mathbf{w}_n \right\rvert^2 + C \lambda_{\delta}^\frac{4}{3} (\alpha) |\mathbf{w}_n |^2 \left\lvert \rA^{\frac{1}{2}} \mathbf{z}_n \right\rvert^\frac{2}{3} |\rA \mathbf{z}_n|^\frac{2}{3} \nonumber\\
&\quad + \dfrac{\alpha^2}{24} |\rA \mathbf{w}_n|^2 + C \lambda_{\delta}^2(\alpha) |\rA\mathbf{z}_n|^2 \left(\alpha^2 \left\lvert \rA^{\frac{1}{2}} \mathbf{w}_n \right\rvert^2 + |\mathbf{w}_n|^2\right). \label{Eq:Est-ProbConv-1}
\end{align}
By using the bilinearity of $\tilde{B}$ and $B$ it is not difficult to see that 
\begin{align*}
& (\lambda_{\delta} (\alpha)  \tilde{B}_\alpha (\mathbf{z}_n, J_\alpha^{-1} \mathbf{z}_n ) - (1 - \delta) B(\mathbf{z}_n, \mathbf{z}_n) , J_\alpha^{-1} \mathbf{w}_n)\\
& =  (\lambda_{\delta} (\alpha)  \tilde{B}_\alpha (\mathbf{z}_n,   \mathbf{z}_n ) - (1 - \delta) B(\mathbf{z}_n, \mathbf{z}_n) , J_\alpha^{-1} \mathbf{w}_n) + \alpha^2 \lambda_{\delta}(\alpha) ( \tilde{B}_\alpha (\mathbf{z}_n, \rA\mathbf{z}_n) , J_\alpha^{-1} \mathbf{w}_n ) \\
& = ( [\lambda_{\delta}(\alpha) - (1 - \delta)] B(\mathbf{z}_n, \mathbf{z}_n), \mathbf{w}_n \rangle + (1 - \delta) \langle J_\alpha B(\mathbf{z}_n, \mathbf{z}_n) - B(\mathbf{z}_n, \mathbf{z}_n), J_\alpha^{-1}(\alpha)\mathbf{w}_n ) \\
&  \hspace{0.5cm}+ \alpha^2 \lambda_{\delta}(\alpha)  (\tilde{B}_\alpha (\mathbf{z}_n, A\mathbf{z}_n), \mathbf{w}_n )\\
& = R_1 + R_2 + R_3.
\end{align*}
Owing to \eqref{2.10} and the Young inequality we get:
\begin{align*}
R_1 &\leq |\lambda_{\delta}(\alpha) - (1 - \delta)| \left\lvert \rA^{\frac{1}{2}} \mathbf{z}_n \right\rvert^{\frac{3}{2}} |\rA\mathbf{z}_n|^{\frac{1}{2}} |\mathbf{w}_n|
\end{align*}
In a similar way,
\begin{align*}
R_2 & \leq (1 - \delta) \alpha^2 \langle B(\mathbf{z}_n, \mathbf{z}_n) , \rA\mathbf{w}_n \rangle \\
& \leq (1 - \delta) \alpha^2 |\mathbf{z}_n|^\frac{1}{2} \left\lvert \rA^{\frac{1}{2}} \mathbf{z}_n \right\rvert |\rA\mathbf{z}_n|^{\frac{1}{2}}  |\rA\mathbf{w}_n| \\
& \leq \dfrac{\alpha^2}{24} |\rA\mathbf{w}_n |^2 + (1 - \delta )^2 \alpha^2 |\mathbf{z}_n| \left\lvert \rA^{\frac{1}{2}} \mathbf{z}_n \right\rvert^2 |\rA\mathbf{z}_n|. 
\end{align*}
As for $R_3$ we use \eqref{2.15} and the Young inequality to obtain
\begin{align*}
R_3 & = \alpha^2 \lambda_{\delta}(\alpha) \langle \tilde{B} (\mathbf{z}_n, \rA\mathbf{z}_n) , \mathbf{w}_n \rangle \\
& \leq \alpha^2 \lambda_{\delta}(\alpha) 	\left\lvert \rA^{\frac{1}{2}} \mathbf{z}_n \right\rvert |\rA\mathbf{z}_n| |\rA\mathbf{w}_n|\\
& \leq \dfrac{\alpha^2}{24} |\rA \mathbf{w}_n|^2 + \alpha^2 \lambda_{\delta}^2(\alpha) \left\lvert \rA^{\frac{1}{2}} \mathbf{z}_n \right\rvert^2 |\rA\mathbf{z}_n|^2.
\end{align*}
Hence
\begin{align}
& (\lambda_{\delta} (\alpha)  \tilde{B}_\alpha (\mathbf{z}_n, J_\alpha^{-1} \mathbf{z}_n ) - (1 - \delta) B(\mathbf{z}_n, \mathbf{z}_n) , J_\alpha^{-1} \mathbf{w}_n) \notag\\
& \leq \dfrac{\alpha^2}{12} |\rA\mathbf{w}_n|^2 + C \alpha^2 [\lambda_{\delta}^2(\alpha) + (1 - \delta)] \left\lvert \rA^{\frac{1}{2}} \mathbf{z}_n \right\rvert^2 |\rA \mathbf{z}_n| [ |\mathbf{z}_n| + |\rA\mathbf{z}_n| ] \notag \\
& \hspace{0.5cm} + |\lambda_{\delta}(\alpha) - (1 - \delta)| \left\lvert \rA^{\frac{1}{2}} \mathbf{z}_n \right\rvert^{\frac{3}{2}} |\rA\mathbf{z}_n|^{\frac{1}{2}} |\mathbf{w}_n|. \label{Eq: Est-ProbConv-2} 
\end{align}
Combining \eqref{Eq:Est-ProbConv-1} and \eqref{Eq: Est-ProbConv-2} we see that 
\begin{equation}
\begin{split}
\lvert( - \mathbf{N} [\mathbf{y}_n, \mathbf{z}_n] , J_\alpha^{-1} \mathbf{w}_n)\rvert \le & \dfrac{3 \alpha^2}{24} |\rA\mathbf{w}_n|^2+ \dfrac{1}{24} \left\lvert \rA^{\frac{1}{2}} \mathbf{w}_n \right\rvert^2+ C \lambda_{\delta}^\frac{4}{3} (\alpha)  \left\lvert \rA^{\frac{1}{2}} \mathbf{z}_n \right\rvert^\frac{2}{3} |\rA \mathbf{z}_n|^\frac{2}{3} [|\mathbf{w}_n |^2 +\alpha^2 \lvert \rA^\frac12 \mathbf{w}_n \rvert^2] \\
& + C \lambda_{\delta}^2(\alpha) |\rA\mathbf{z}_n|^2 \left(|\mathbf{w}_n|^2+\alpha^2 \left\lvert \rA^{\frac{1}{2}} \mathbf{w}_n \right\rvert^2 \right). \\
+& |\lambda_{\delta}(\alpha) - (1 - \delta)|\left( \left\lvert \rA^{\frac{1}{2}} \mathbf{z}_n \right\rvert |\rA\mathbf{z}_n| + \lvert \rA^\frac12 \mathbf{z}_n \rvert^2\left[|\mathbf{w}_n|^2+ \alpha^2 \lvert \rA^\frac12 \mathbf{w}_n \rvert^2\right] \right)\\
+ & C \alpha^2 [\lambda_{\delta}^2(\alpha) + (1 - \delta)] \left\lvert \rA^{\frac{1}{2}} \mathbf{z}_n \right\rvert^2 |\rA \mathbf{z}_n| [ |\mathbf{z}_n| + |\rA\mathbf{z}_n| ] 
\label{Eq: Est-ProbConv-2-B} 
\end{split}
\end{equation}
Our next task is to estimate
\begin{align*}
 ( \mathbf{L}[\mathbf{y}_n, \mathbf{z}_n], J_\alpha^{-1} \mathbf{w}_n)= & \delta (\tilde{B}_\alpha(\bu , J_\alpha^{-1} \mathbf{y}_n) - B(\bu, \mathbf{z}_n) , J_\alpha^{-1} \mathbf{w}_n)  +  \delta (\tilde{B}_\alpha (\mathbf{y}_n, J_\alpha^{-1} \bu  ) - B( \mathbf{z}_n, \bu) , J_\alpha^{-1} \mathbf{w}_n) \\
& =: I+ L.
\end{align*}
Using the bilinearity of $\tilde{B}$ and \eqref{Eq:Def-Of-tildeB} we see that
\begin{align*}
I & = \delta \langle \tilde{B} (\bu, J_\alpha^{-1} \mathbf{w}_n), \mathbf{w}_n  \rangle_{D(\rA)'} + \delta  (\tilde{B}_\alpha(\bu, J_\alpha^{-1} \mathbf{z}_n) - B(\bu, \mathbf{z}_n) , J_\alpha^{-1} \mathbf{w}_n)  \\
& = -\delta (\tilde{B}(\mathbf{w}_n, \mathbf{w}_n) , \bu )+ \delta \alpha^2 \langle \tilde{B}(\bu, \rA \mathbf{w}_n) , \mathbf{w}_n\rangle_{D(\rA)'}
+ \delta (\tilde{B}_\alpha (\bu , J_\alpha^{-1} \mathbf{z}_n)  - B(\bu, \mathbf{z}_n) , J_\alpha^{-1} \mathbf{w}_n).
\end{align*}
By denoting  $I_1$ and $I_2$ the first two terms on the right hand side of the above equation and using the bilinearity of $\tilde{B}$ again and \eqref{Eq:Skew-Symm-tildeB} we find that
\begin{align*}
I = I_1 + I_2 + \delta \alpha^2 \langle \tilde{B}(\bu, \rA \mathbf{z}_n) , \mathbf{w}_n \rangle_{D(\rA)'}  - \delta \alpha^2 \langle B(\bu, \mathbf{z}_n) , \rA\mathbf{w}_n \rangle_{D(\rA)'} -\delta  (B(\mathbf{w}_n, \mathbf{z}_n) , \bu ) .
\end{align*}
In a similar way we can show that
\begin{align*}
L &=  \delta  (\tilde{B}_\alpha(\mathbf{y}_n, J_\alpha^{-1}\bu) - B(\mathbf{z}_n, \bu), J_\alpha^{-1}\mathbf{w}_n)  \\
&=    \delta \alpha^2 \langle \tilde{B}(\mathbf{z}_n,\rA \bu) , \mathbf{w}_n \rangle_{D(\rA)'} + \delta \alpha^2 \langle B(\mathbf{z}_n, \bu), \rA \mathbf{w}_n \rangle_{D(\rA)'}+\delta ( B(\mathbf{w}_n, \mathbf{z}_n) , \bu ) .
\end{align*}
Hence
\begin{align*}
 I + L 
& = I_1 + I_2 + \delta \alpha^2 \langle \tilde{B}(\bu, \rA \mathbf{z}_n) , \mathbf{w}_n \rangle_{D(\rA)'}  - \delta \alpha^2 ( B(\bu, \mathbf{z}_n) , \rA\mathbf{w}_n )\\
& \hspace{0.5cm}+ \delta \alpha^2 \langle \tilde{B}(\mathbf{z}_n,  \rA\bu) , \mathbf{w}_n \rangle_{D(\rA)'} + \delta \alpha^2 \langle B(\mathbf{z}_n, \bu), \rA \mathbf{w}_n \rangle_{D(\rA)'}\\
& =: \sum_{i = 1}^{6} I_i.
\end{align*}
In the next lines we will estimate $I_i$, $i = 1, \ldots, 6$.\\

Using \eqref{Eq:skew-symmetry}, \eqref{refine-2.10}  and the Young inequality we see that
\begin{align*}
I_1  = \delta  (B(\mathbf{w}_n, \bu), \mathbf{w}_n )
& \leq \delta |\mathbf{w}_n|^{\frac{1}{2}} \left\lvert \rA^{\frac{1}{2}} \mathbf{w}_n \right\rvert^{\frac{1}{2}} \left\lvert \rA^{\frac{1}{2}} \bu \right\rvert^{\frac{1}{2}} |\rA\bu|^{\frac{1}{2}} |\mathbf{w}_n| \\
& \leq \dfrac{1}{24} \left\lvert \rA^{\frac{1}{2}} \mathbf{w}_n \right\rvert^2 + C \delta^2 |\mathbf{w}_n|^2 \left\lvert \rA^{\frac{1}{2}} \bu \right\rvert^{\frac{2}{3}} |\rA \bu|^{\frac{2}{3}}.
\end{align*}
In a similar fashion we can show that
\begin{align*}
I_2 & = \delta ( B(\bu, \mathbf{w}_n) , \rA \mathbf{w}_n ) + \delta \alpha^2 ( B(\mathbf{w}_n, \bu) , \rA \mathbf{w}_n)  \\
& \leq \delta \alpha^2 |\rA\mathbf{w}_n| \left[ |\bu|^{\frac{1}{2}} \left\lvert \rA^{\frac{1}{2}} \bu \right\rvert^{\frac{1}{2}} \left\lvert \rA^{\frac{1}{2}} \mathbf{w}_n  \right\rvert^{\frac{1}{2}} \left\lvert \rA \mathbf{w}_n \right\rvert^{\frac{1}{2}}  + |\mathbf{w}_n|^{\frac{1}{2}} \left\lvert \rA^{\frac{1}{2}} \mathbf{w}_n \right\rvert^{\frac{1}{2}} \left\lvert \rA^{\frac{1}{2}} \bu \right\rvert^{\frac{1}{2}} \left\lvert \rA  \bu \right\rvert^{\frac{1}{2}}\right] \\
& \leq \dfrac{\alpha^2}{96} |\rA \mathbf{w}_n|^2  + C \delta^2 \alpha^2 \left[  |\bu|^2 \left\lvert \rA^{\frac{1}{2}} \bu \right\rvert^2  \left\lvert \rA^{\frac{1}{2}} \mathbf{w}_n \right\rvert^{2}  + |\mathbf{w}_n|^2 \left\lvert \rA^{\frac{1}{2}} \mathbf{w}_n \right\rvert \left\lvert \rA^{\frac{1}{2}} \bu \right\rvert |\rA \bu| \right]\\
& \leq \dfrac{\alpha^2}{96} |\rA \mathbf{w}_n|^2  + C \delta^2 \alpha^2 \left[ \left(1 + |\bu|^2 \left\lvert \rA^{\frac{1}{2}} \bu \right\rvert^2 \right) \left\lvert \rA^{\frac{1}{2}} \mathbf{w}_n \right\rvert^{2}  + |\mathbf{w}_n|^2 \left\lvert \rA^{\frac{1}{2}} \bu \right\rvert^2 |\rA \bu|^2\right].
\end{align*} 
We now proceed to the estimate of $I_3$.
By using \eqref{2.15} and Young's inequality 
\begin{align*}
I_3 & = \delta \alpha^2 \langle \tilde{B} (\bu, \rA \mathbf{z}_n) , \mathbf{w}_n \rangle_{D(\rA)'} \\
& \leq C \delta \alpha^2 |\rA \mathbf{z}_n| \left\lvert \rA^{\frac{1}{2}} \bu \right\rvert |\rA \bu| \\
& \leq C \delta^2 \alpha^2 |\rA  \mathbf{z}_n|^2 \left\lvert \rA^{\frac{1}{2}} \bu \right\rvert + \dfrac{\alpha^2}{96} |\rA \mathbf{w}_n|^2.
\end{align*}
In a similar way we show that 
\begin{align*}
I_5 & = \delta \alpha^2 \langle \tilde{B} (\mathbf{z}_n, \rA \bu), \mathbf{w}_n \rangle_{D(\rA)'} \\
& \leq C \delta^2 \alpha^2 |\rA  \mathbf{u}|^2 \left\lvert \rA^{\frac{1}{2}} \mathbf{z}_n \right\rvert + \dfrac{\alpha^2}{96} |\rA \mathbf{w}_n|^2.
\end{align*}
Finally by using \eqref{2.9}, \eqref{2.10} and the Young inequality, the term $I_4 + I_6$ can be estimated as follows
\begin{align*}
I_4 + I_6 & \leq \delta \alpha^2 |\rA \mathbf{w}_n| \left\lvert \rA^{\frac{1}{2}} \bu \right\rvert^{\frac{1}{2}} |\rA \bu|^{\frac{1}{2}} \left\lvert \rA^{\frac{1}{2}} \mathbf{z}_n \right\rvert\\
& \leq \dfrac{\alpha^2}{96} |\rA \mathbf{w}_n|^2 + \left\lvert \rA^{\frac{1}{2}} \mathbf{z}_n \right\rvert^2 \left\lvert \rA^{\frac{1}{2}} \bu \right\rvert |\rA \bu|.   
\end{align*}
Thus,
\begin{align}
&I+L= \langle \mathbf{L}[\mathbf{y}_n, \mathbf{z}_n], J_\alpha^{-1} \mathbf{w}_n\rangle  \\
&
 \leq \dfrac{\alpha^2}{24} | \rA \mathbf{w}_n|^2 + \dfrac{1}{24} \left\lvert \rA^{\frac{1}{2}} \mathbf{w}_n \right\rvert^2 + C  \delta^2 \left( |\mathbf{w}_n|^2 + \alpha^2 \left\lvert \rA^{\frac{1}{2}} \mathbf{w}_n \right\rvert^2  \right) \left( 1 + \left\lvert \rA^{\frac{1}{2}} \bu \right\rvert^2 |\rA \bu|^2 \right) \notag \\
& \hspace{0.5cm} + C \delta^2 \alpha^2 \left[ \left\lvert \rA^{\frac{1}{2}} \bu \right\rvert^2 \left( |\rA \mathbf{z}_n|^2 \left\lvert \rA^{\frac{1}{2}} \mathbf{z}_n \right\rvert^2 \right) + |\rA \bu|^2 \left\lvert \rA^{\frac{1}{2}} \mathbf{z}_n \right\rvert^2 \right]   \label{Eq: Est-ProbConv-3}
\end{align}  
We now deal with the control terms. It is not difficult to prove that
\begin{align*}
& (\G(\Phi_n)h_n - G(\Phi) h_n, J_\alpha^{-1} \mathbf{w}_n  )\notag \\
& = ( J_\alpha G(\Psi_n) h_n - J_\alpha G(\Phi_n) h_n , J_\alpha^{-1} \mathbf{w}_n )+ ( J_\alpha G(\Phi_n) h_n -  G(\Phi) h_n , J_\alpha^{-1} \mathbf{w}_n) \notag \\
& \leq | G(\Psi_n) h_n -  G(\Phi) h_n | | \mathbf{w}_n| + \langle J_\alpha G(\Phi_n) h_n -  G(\Phi_n) h_n , J_\alpha^{-1} \mathbf{w}_n \rangle.
\end{align*} 
Using the Assumption \ref{Assum-Uniqness} and the definitions of $\Psi_n$ and $\Phi_n$ yield
\begin{align*}
&( \G(\Phi_n)h_n - G(\Phi) h_n, J_\alpha^{-1} \mathbf{w}_n  )\\
& \leq |\Psi_n - \Phi_n| \Vert h_n\Vert_{\rK} |\mathbf{w}_n| + | ( J_\alpha G(\Phi_n) h_n -  G(\Phi_n) h_n , J_\alpha^{-1} \mathbf{w}_n ) | \\
& \leq C [\lambda_{\delta}(\alpha) |\mathbf{w}_n|^2+ \lvert \lambda_\delta(\alpha)-(1-\delta\rvert ) \lvert \mathbf{z}_n\rvert \lvert \mathbf{w}_n\rvert ]\Vert h_n\Vert_{\rK} + | ( J_\alpha G(\Phi_n) h_n -  G(\Phi_n) h_n , J_\alpha^{-1} \mathbf{w}_n )|. 
\end{align*}
Let us now deal with the second term on the right hand side of the last inequality. Thanks to Assumption \ref{Assum-Uniqness}, inequality \eqref{Eq:3.20} and the Young inequality we have
\begin{align*}
&|  (J_\alpha G(\Phi_n) h_n -  G(\Phi_n) h_n , J_\alpha^{-1} \mathbf{w}_n)  | \\
& \leq \alpha | (\alpha^2 \rA)^{\frac{1}{2}} J_\alpha \rA^{\frac{1}{2}} G(\Phi_n) h_n | |J_{\alpha}^{-1} \mathbf{w}_n| \\
& \leq C \alpha (1 + | \Phi_n|)\Vert h_n\Vert_{\rK} |J_\alpha^{-1} \mathbf{w}_n| \\
& \leq C \alpha^2 (1 + \delta^2 |\bu|^2 + (1-\delta)^2 |\mathbf{z}_n|^2) \Vert h_n\Vert_{\rK}^2 + C \Vert h_n\Vert_{\rK}^2 |\mathbf{w}_n|^2 + \dfrac{\alpha^2}{24} |\rA \mathbf{w}_n|^2. 	
\end{align*}     
Thus,
\begin{equation}
\begin{split}
& (\G (\Psi_n)h_n - G (\Phi_n)h_n , J_\alpha^{-1} \mathbf{w}_n ) \notag \\
 \leq & C (1 + \lambda_{\delta}(\alpha))  \Vert h_n\Vert_{\rK}^2  [	|\mathbf{w}_n|^2 + \alpha^2 \lvert \rA^\frac12 \mathbf{w}_n\rvert^2]+ \dfrac{\alpha^2}{24} |\rA \mathbf{w}_n|^2  \\ 
& +  C \alpha^2 (1 + \delta^2 |\bu|^2 + (1-\delta)^2 |\mathbf{z}_n|^2) \Vert h_n\Vert_{\rK}^2
+ \lvert \lambda_\delta(\alpha)-(1-\delta)\rvert^2 \lvert \mathbf{z}_n\rvert^2
.\label{Eq: Est-ProbConv-4}
\end{split}
\end{equation}
By using  Assumption \ref{Assum-Uniqness} and the definition of the stopping time $\tau_N$, it is not difficult to show that
\begin{align}
\alpha \lambda_{\delta}^{-2}(\alpha) \int_{0}^{t \wedge \tau_N}  \Vert G(\Psi_n)\Vert_{\mathscr{L}(\rK, \h)}^2 ds \leq&  C \alpha \lambda_{\delta}^{-2} (\alpha)\int_{0}^{t \wedge \tau_N}  (1 + \delta^2 |\bu|^2 + \lambda_{\delta}^2(\alpha) |\mathbf{y}_n|^2) ds\notag \\
\leq &C \alpha \lambda_{\delta}^{-2} (\alpha) T  (1 + \delta^2 \sup_{s\in [0,T]}|\bu(s) |^2 + \lambda_{\delta}^2(\alpha) N ) .\label{Eq: Est-ProbConv-5}
\end{align}

Before proceeding further we set 
\begin{align*}
\mathrm{Y}_n=& \left(\lambda_\delta^2(\alpha)  + \lvert \lambda_\delta(\alpha)- (1-\delta)\rvert + \lambda_\delta^\frac43(\alpha)  \right)\left(1+ \lvert \rA\mathbf{z}_n\rvert^2\right)+ \delta^2(1+ \lvert \rA^\frac12 \bu \rvert^2\lvert \rA \bu \rvert^2)\\ & + (1+ \lambda_\delta(\alpha))\lVert h_n \rVert^2_{\mathrm{K}}.
\end{align*}
\begin{align*}
\mathrm{R}_n= &  \left( \lvert \lambda_\delta(\alpha)- (1-\delta)\rvert + \alpha^2 [ \lambda_\delta^2(\alpha)+ (1-\delta)] \lvert \rA^\frac12 \mathbf{z}_n\rvert^2+ \delta^2 \alpha^2 \lvert \rA^\frac12 \bu \lvert^2 \lvert \rA^\frac12 \mathbf{z}_n \rvert^2 \right) \lvert \rA\mathbf{z}_n \rvert^2 \\
&+ \left(\alpha^2 +\alpha^2 \lambda_\delta^{-1}(\alpha) \lvert \rA^\frac12 \bu \rvert^2 + \delta^2 \alpha^2 \lvert \rA^\frac12 \mathbf{z}_n \rvert^2 \right) + \lvert \lambda_\delta(\alpha)-(1-\delta)\rvert^2 \lvert \mathbf{z}_n\rvert^2\\
& + \alpha^2  \left(1+\delta^2 \sup_{s\in [0,T]}\lvert \bu(s) \rvert^2  + (1-\delta)^2  \sup_{s\in [0,T]}|\mathbf{z}_n(s|^2 \right) \lVert h_n \rVert^2_{\mathrm{K}} \\
& + \alpha \lambda_{\delta}^{-2} (\alpha) T  (1 + \delta^2 \sup_{s\in [0,T]}|\bu(s) |^2 + \lambda_{\delta}^2(\alpha) N ) .
\end{align*}
 
Then, by plugging \eqref{Eq:Est-ProbConv-0}, \eqref{Eq: Est-ProbConv-2-B}, \eqref{Eq: Est-ProbConv-3}, \eqref{Eq: Est-ProbConv-4}, \eqref{Eq: Est-ProbConv-5} and \eqref{Eq: Est-ProbConv-6} into \eqref{Eq:Conv-Prob-Start} and using the Sobolev embedding $D(\rA)\subset D(\rA^\frac12)\subset \h$ we obtain that there exist  constants $C_0, C_1>0$ such that with probability $1$ and  for all $n \in \mathbb{N}$  
\begin{align}
&\Vert \mathbf{w}_n (t \wedge \tau_N)\Vert_\alpha^2 + \int_{0}^{t \wedge \tau_N} \left[ \frac12 \left\lvert \rA^{\frac{1}{2}}  \mathbf{w}_n (s)\right\rvert^2  + \frac{\alpha^2}2 \left\lvert \rA   \mathbf{w}_n (s)\right\rvert^2 \right] ds \notag\\
&\le \lVert \mathbf{w}_n(0) \rvert^2_\alpha+  C_0  \int_{0}^{t \wedge \tau_N}  \mathrm{Y}_n(s) \lVert \mathbf{w}_n(s) \rVert^2_\alpha ds  + C_1  \int_0^{t \wedge \tau_N} R_n(s) ds+ \alpha^{\frac{1}{2}} \lambda^{-1}_{\delta}(\alpha)  \mathscr{M}_n(t\wedge \tau_N),\label{Eq:Conv-Prob-Start1}
\end{align}
where 
\begin{equation*}
\mathscr{M}_n(t)=\int_{0}^{t} ( J_\alpha^{-1} \mathbf{w}_n, \G (\Psi_n) dW), \; t\in [0,T].
\end{equation*}
We now deal with the stochastic term. By using the Burkholder-Davis-Gundy inequality, the Assumption \ref{Assum-Uniqness} and the Young inequality we deduce that for any $\theta>0$ there exist two constant $C_2, c_2>0$ such that for all $n \in \mathbb{N}$
\begin{align}
\lqq{ \alpha^{\frac{1}{2}} \lambda^{-1}_{\delta}(\alpha) \mathbb{E} \sup_{0\le s \le t}\left\lvert \int_{0}^{ s \wedge \tau_N} ( J_\alpha^{-1} \mathbf{w}_n, \G(\Psi_n) dW)  \right\rvert} \notag\\
& = \alpha^{\frac{1}{2}} \lambda^{-1}_{\delta}(\alpha) \mathbb{E} \sup_{0\le s \le t} \left \lvert \int_{0}^{ s \wedge \tau_N} ( \mathbf{w}_n , G(\Psi_n) dW ) \right \rvert \notag \\
& \leq c_2 \alpha^{\frac{1}{2}} \lambda^{-1}_{\delta}(\alpha) \mathbb{E} \left[\int_{0}^{t \wedge \tau_N} \lVert \mathbf{w}_n\rVert^2_\alpha  \Vert G(\Psi_n)\Vert_{\mathscr{L}(\rK, \h)}^2 ds \right]^{\frac{1}{2}} \notag\\
& \leq \theta \mathbb{E} \sup_{s\in [0,t]} \lVert \mathbf{w}_n (s\wedge \tau_N) \rVert^2_\alpha+    C_2 \alpha\lambda^{-2}_{\delta}(\alpha) \mathbb{E} \int_{0}^{t \wedge \tau_N}   (1 + \delta^2|\bu|^2 + \lambda_{\delta}^2(\alpha) |\mathbf{y}_n|^2) ds . \notag
\end{align}  
Using the definition of the stopping time $\tau_N$ yields
\begin{align}
& \alpha^{\frac{1}{2}} \lambda^{-1}_{\delta}(\alpha) \mathbb{E} \sup_{0\le s \le t}\left\lvert \int_{0}^{ s \wedge \tau_N} \langle J_\alpha^{-1} \mathbf{w}_n, \G(\Psi_n) dW \rangle \right\rvert \nonumber \\
& \leq \theta \mathbb{E} \sup_{s\in [0,t]} \lVert \mathbf{w}_n (s\wedge \tau_N) \rVert^2_\alpha+    C_2 \alpha\lambda^{-2}_{\delta}(\alpha) \mathbb{E} \int_{0}^{t \wedge \tau_N} (1 + \delta^2|\bu|^2 + \lambda_{\delta}^2(\alpha) N ) ds .   \label{Eq: Est-ProbConv-6}
\end{align}
Next observe that thanks to the  estimates \eqref{Eq:Est-Det-Control}-\eqref{Eq:Est-Det-NSE}, the fact $\int_0^T \lVert h_n \rVert^2_{\mathrm{K}}\le M $  we see that there exists a deterministic constant $c_3>0$ such that with probability $1$
\begin{equation}
e^{\int_0^{T} C_0 \mathrm{Y}_n(s) ds    } \le e^{(T+c_3)\left(\lambda_\delta^2(\alpha)  + \lvert \lambda_\delta(\alpha)- (1-\delta)\rvert + \lambda_\delta^\frac43(\alpha)  \right)+ \delta^2c_3 (1+ c_3 )M + \alpha^2 \lambda_\delta(\alpha)T. }
\end{equation}
In a similar way, we can show that there exists a deterministic constant $C_3>0$, which may depend on $M$, $N$ and $T$,  such that with probability $1$
\begin{equation}
\int_0^{T\wedge \tau_n} \mathrm{R}_n(s) ds\le C_3 \Sigma_n ,
\end{equation}
where the sequence $\Sigma_n,\; n \in \mathbb{N}$ is defined by
\begin{equation}
\begin{split}
\Sigma_n= &  \lvert \lambda_\delta(\alpha_n)- (1-\delta)\rvert + \alpha_n^2 [ \lambda_\delta^2(\alpha_n)+ (1-\delta)] + \delta^2 \alpha_n^2  + \alpha_n^2 \lambda_\delta^{-1}(\alpha_n) + \alpha_n \lambda_{\delta}^{-2} (\alpha_n) +\alpha_n^2 +\alpha_n 
\end{split}
\end{equation}
Next, since $\lambda_\delta(\alpha_n)\le 1$,   $\lvert \lambda_\delta(\alpha_n)- (1-\delta)\rvert \to 0$, and $\alpha_n^2\lambda_\delta(\alpha_n)\to 0$ as $n \to \infty$, we deduce that there exists a deterministic constant $C_4>0$ such that with probability $1$
\begin{equation}
\sup_{n \in \mathbb{N}}e^{\int_0^{T} C_0 \mathrm{Y}_n(s) ds    } \le  C_4. 
\end{equation}
Thus, by choosing $\theta >0$ so that $2 \theta C_4\le 1$  and applying the version of Gronwall's lemma given in \cite[Lemma A.1]{Chueshov+Millet}
we obtain that for all $t\in [0,T]$ and $n \in \mathbb{N}$
\begin{align*}
\mathbb{E} \sup_{s\in [0,t]}\Vert \mathbf{w}_n (t \wedge \tau_N)\Vert_\alpha^2 +\mathbb{E} \int_{0}^{t \wedge \tau_N}  \left[ \frac12 \left\lvert \rA^{\frac{1}{2}}  \mathbf{w}_n (s)\right\rvert^2  + \frac{\alpha^2}2 \left\lvert \rA   \mathbf{w}_n (s)\right\rvert^2 \right] ds \notag
\le  C_4C_3 \Sigma_n, 
\end{align*}
which implies 
\begin{align*}
\mathbb{E} \sup_{s\in [0,t]}\lvert \mathbf{w}_n (t \wedge \tau_N)\rvert ^2 +\mathbb{E} \int_{0}^{t \wedge \tau_N}  \frac12 \left\lvert \rA^{\frac{1}{2}}  \mathbf{w}_n (s)\right\rvert^2  ds \le  C_4C_3 \Sigma_n.
\end{align*}
	Now, since $\lambda_\delta(\alpha_n)\le 1$,   $\lvert \lambda_\delta(\alpha_n)- (1-\delta)\rvert \to 0$ and $\alpha_n \lambda_\delta^{-\ell}(\alpha_n) \to 0$, $\ell\in \{1,2\}$, as $n \to \infty$ we infer that
\begin{equation}\label{CONV-OMEGA-DA}
\mathbb{E} \sup_{r\in [0,T]}\lvert \mathbf{w}_n(r\wedge \tau_N )\rvert^2 + \mathbb{E} \int_0^{T\wedge \tau_N }\lvert \rA^\frac12 \mathbf{w}_n(s)\rvert^2 \, ds \to 0 \text{ as } n\to \infty. 
\end{equation}

Next, 
let $\gamma>0$ and $\eps>0$ be arbitrary numbers. Let us set 
\begin{equation*}
X_n(T) =\sup_{r\in [0,T]}\lvert \mathbf{w}_n(r)\rvert^2 + \int_0^T \lvert \rA^\frac12 \mathbf{w}_n(s)\rvert^2 \, ds.
\end{equation*}
Then,  it is not difficult to check that 
\begin{align}
\mathbb{P}\left(X_n(T) \ge \eps\right)\le& \mathbb{P}(\sup_{r\in[0,T]}\lvert X_n(T) , \tau_N=T )+ \mathbb{P}(\sup_{r\in [0,T]}\lvert \mathbf{y}_n(r)\rvert^2\ge N )\nonumber \\
& \le \frac1\eps \mathbb{E} X_n(T\wedge \tau_N) + \frac1N \mathbb{E} \sup_{r\in [0,T]}(\lvert \mathbf{y}_n(r)\rvert ). \label{EST-CONV-PROB-DA}
\end{align}
Owing to estimate \eqref{Estim} one can find $N_0>0$ such that if $N\ge N_0$ then  $$\frac1N \mathbb{E} \sup_{r\in [0,T]}(\lvert \mathbf{y}_n(r)\rvert^2 )< \frac{\gamma}{2}.$$ Thus, thanks to \eqref{CONV-OMEGA-DA} and \eqref{EST-CONV-PROB-DA} we infer that there exists $n_0\in \mathbb{N}$ such that for all $n\ge n_0$ 
\begin{equation*}
\mathbb{P}\left(\left[\sup_{r\in [0,T]}\lvert \mathbf{w}_n(r )\rvert^2 + \int_0^{T}\lvert \rA^\frac12 \mathbf{w}_n(s)\rvert^2 \, ds\right] \ge \eps \right)<\gamma,
\end{equation*}
which completes the proof of Proposition \ref{Prop:Conv-In-Proba}.
\end{proof}
We will also need the following result.
\begin{Lem}\label{Lem-Prop-Imp-2}
		Let $M>0$, $(h_n)_{n \in \mathbb{N}}\subset \mathscr{A}_M$, $h\in \mathscr{A}_M$, and $(\alpha_n)_{n \in \mathbb{N}}\subset (0,1]$ be a sequence converging to 0. Also, let $\delta \in \{0,1\}$ and $\xi \in D(\rA^\frac12)$. Let us assume that Assumption \ref{Assum-Uniqness} holds. Let $h_n$ be  a sequence converging in distribution to $h$ as $\mathcal{A}_M$-valued random variable. 
		
		Then,  the process $\Gamma^{0,\delta}_\xi \left(\int_0^{\cdot} h_n(r)dr \right)$ converges in distribution to $\Gamma^{0,\delta}_\xi \left(\int_0^{\cdot} h(r)dr \right)$ as $C([0,T];\h)\cap \el^2(0,T;D(\rA^{\frac 12}))$-valued random variables.
\end{Lem}
\begin{proof}[Proof of Lemma \ref{Lem-Prop-Imp-2}]
	Before diving into the depth of the proof we recall that $\mathcal{A}_M$ is a Polish space when endowed with the metric defined in \eqref{Eq:Metric-SM}. Now, since, by assumption, $h_n \to h$ in law as $\mathcal{A}_M$-valued random variables, we can infer from the Skorokhod's theorem that one can find a probability space $(\bar{\Omega}, \bar{\mathscr{F}},\bar{\mathbb{P}})$ on which there exist $\mathcal{A}_M$-valued random variables $\bar{h}_n$, $\bar{h}$ having the same laws as $h_n$ and $h$, respectively, and satisfying 
	\begin{equation}
	\bar{h}_n \to h \text{ in }  \mathcal{A}_M, \bar{\mathbb{P}}-\text{a.s.}. 
	\end{equation} 
	From the last property and Theorem \ref{THM:LDP1}   we derive that 
	\begin{equation}\label{CONV-COTW}
	\Gamma^{0,\delta}_\xi \left(\int_0^{\cdot} \bar{h}_n(r)dr \right)  \to \Gamma^{0,\delta}_\xi \left(\int_0^{\cdot} \bar{h}(r)dr \right) \text{ in } C([0,T];\h)\cap \el^2(0,T; D(\rA^\frac12)) \;\; \bar{\mathbb{P}}-\text{a.s.}.
	\end{equation}
	Observe that  Theorem \ref{THM:LDP1}  implies in particular that  $\Gamma^{0,\delta}_\xi:  \mathcal{A}_M \to C([0,T];\h)\cap \el^2(0,T; D(\rA^\frac12))$ is continuous. Hence, from the equality of the laws of $h_n$ (resp. $h$) and $\bar{h}_n$ (resp. $\bar{h}$) we infer that  the laws of $\Gamma^{0,\delta}_\xi \left(\int_0^{\cdot} \bar{h}_n(r)dr \right)$  and $\Gamma^{0,\delta}_\xi \left(\int_0^{\cdot} \bar{h}(r)dr \right) $ are equal to the laws of $\Gamma^{0,\delta}_\xi \left(\int_0^{\cdot} {h}_n(r)dr \right)$ and $\Gamma^{0,\delta}_\xi \left(\int_0^{\cdot} {h}(r)dr \right) $, respectively. This observation and the convergence \eqref{CONV-COTW}  complete the proof of Lemma \ref{Lem-Prop-Imp-2}.  
\end{proof}
The next result that we need is contained in the following theorem.
\begin{Thm}\label{Prop-Imp-2}
	Let $M>0$, $(h_n)_{n \in \mathbb{N}}\subset \mathscr{A}_M$, $h\in \mathscr{A}_M$, and $(\alpha_n)_{n \in \mathbb{N}}\subset (0,1]$ be a sequence converging to 0. Also, let $\delta \in \{0,1\}$ and $\xi \in D(\rA^\frac12)$. \newline
	If Assumption \ref{Assum-Uniqness} holds and $h_n$ is a sequence converging in distribution to $h$ as $\mathcal{A}_M$-valued random variable, then the process $\Gamma^{\alpha_n,\delta}_\xi \biggl(W + \alpha_n^{-\frac12} \lambda_\delta(\alpha_n )\int_0^{\cdot} h_n(r)dr\biggr) $ converges in distribution to $\Gamma^{0,\delta}_\xi \left(\int_0^{\cdot} h(r)dr  \right)$ as $C([0,T];\h)\cap \el^2(0,T;D(\rA^{\frac12}))$-valued random variables.
\end{Thm}
\begin{proof}
	Theorem \ref{Prop-Imp-2} readily follows from \cite[Theorem 11.3.3]{Dudley}, Lemmata \ref{Lem-Prop-Imp-1} and \ref{Lem-Prop-Imp-2}.
\end{proof}
\subsection{Proof of Theorem \ref{Main-Theorem-Delta} }\label{Sec:Proof-of-Main-Thm}
In this subsection we will give the proof of our main results which are contained in Theorem \ref{Main-Theorem-Delta}. The proof relies on  a LDP result which follows from \cite[Theorem 3.6 and Theorem 4.4]{BuDu}. We first recall this LDP results. 

 Let $\mathrm{K}$, $\mathrm{K}_1$ be two separable Hilbert spaces and $W$ a Wiener process as in Subsection \ref{Assum-Noise}. We recall that $\mathscr{A}$ is the set of all $\mathrm{K}$-valued predictable process $h$ such that 
\begin{equation}
\mathbb{P}\left(\int_0^T \lVert h(r)\rVert^2_{\mathrm{K}} dr <\infty \right)=1.
\end{equation}
We recall the following result which is exactly \cite[Theorem 3.6]{BuDu}. 
\begin{Thm}\label{BD-Theorem-3.6}
	Let $\Gamma: C([0,T];\mathrm{K})\to \mathbb{R}$ be a bounded, Borel measurable function. Then
	\begin{equation}
	-\log \mathbb{E}e^{-\Gamma(W)}=\inf_{h \in \mathscr{A}} \mathbb{E} \biggl\{\frac12 \int_0^T \lVert h(r)\rVert^2_{\mathrm{K}} + \Gamma\left(W+ \int_0^{\cdot}h(r)dr\right ) \biggr\}.
	\end{equation}
\end{Thm}
Next, let $\mathcal{E}$ be a Polish space, $(\Psi^\eps)_{\eps\in (0,1]}$ a family of Borel measurable maps from $C([0,T];\mathrm{K})$ onto $\mathcal{E}$, and $(X^\eps)_{\eps\in (0,1]}$ a family of $\mathcal{E}$-valued random variables.We have the following result which can be proved by using Theorem \ref{BD-Theorem-3.6} and the idea in the proof of \cite[Theorem 4.4]{BuDu}.
\begin{Thm}\label{BD-Theorem-4.4}
	Let $\varrho$ be a real-valued function defined on $(0,\infty)$ such that $$\varrho(\eps) \to \infty \text{ as } \eps \to 0 .$$
	Assume that there exists a Borel measurable map $\Psi^0: C([0,T];\mathrm{K}) \to \mathcal{E}$ such that the following hold:
	\begin{enumerate}[label=(\textbf{A}\arabic{*})]
		\item \label{BD-Assum1}  if $(h_\eps)_{\eps \in (0,1]}\subset \mathscr{A}_M$, $M>0$, converges in distribution to $h\in \mathscr{S}_M$ as $\mathscr{A}_M$-valued random variables, then   $\Psi^\eps(W +  \varrho(\eps) \int_0^{\cdot} h_\eps(r)dr )$ converges in distribution to $\Psi^0(\int_0^{\cdot} h(r)dr)$.
		\item \label{BD-Assum2} For every $M>0$ the set $K_M= \{\Psi^0(\int_0^{\cdot}h(r) dr): h \in \mathscr{A}_M\}$ is a compact subset of $\mathcal{E}$.
	\end{enumerate}
	Then, the family $(X^\eps)_{\eps \in (0,1]}$ satisfies an LDP with speed $\varrho^2(\eps)$ and rate function $I$ given by 
	\begin{equation}
	I(x)= \inf_{\{ h \in \el^2(0,T;\mathrm{K}):\;\; x= \Psi^0(\int_0^{\cdot}h(r)dr )\} } \biggl\{\frac12 \int_0^T \lVert h(r)\rVert^2_{\mathrm{K}}  \biggr \}.
	\end{equation}
\end{Thm}
Now, we are ready to  give the promised proof of our main theorem.
\begin{proof}[Proof of Theorem \ref{Main-Theorem-Delta}]
	Owing to Theorems \ref{THM:LDP1}  and  \ref{Prop-Imp-2} the assumptions \ref{BD-Assum1} and \ref{BD-Assum2} of Theorem \ref{BD-Theorem-4.4} are satisfied on $\mathcal{E}=C([0,T];\h)\cap \el^2(0,T;D(\rA^{\frac 12}))$. Thus, we infer that for $\delta\in \{0,1\}$ the solution $\bu^{\alpha,\delta}$ to \eqref{Eq: ABS-SVE} satisfies an LDP on $\mathcal{E}$ with speed $\alpha^{-1}\lambda^2_\delta(\alpha)$ and rate function $I_\delta$. This completes the proof of Theorem \ref{Main-Theorem-Delta}. 
\end{proof}


\begin{thebibliography}{99}
	
	\bibitem{Bessaih+Millet-2009}
	Bessaih, H.~ and Millet, A.~
	\newblock Large deviation principle and inviscid shell models.
	\newblock {\em Electron. J. Probab.}, 14(89):2551--2579, 2009.
	
	\bibitem{Bessaih+Millet-2012}
	Bessaih, H.~ and Millet, A.~
	\newblock Large deviations and the zero viscosity limit for 2{D} stochastic
	{N}avier-{S}tokes equations with free boundary.
	\newblock {\em SIAM J. Math. Anal.}, 44(3):1861--1893, 2012.
	
	\bibitem{ZB+BG+TJ-2017}
	Brze{\'{z}}niak, Z.~, Goldys, B.~ and Jegaraj, T.~
	\newblock Large deviations and transitions between equilibria for stochastic
	landau--lifshitz--gilbert equation.
	\newblock {\em Archive for Rational Mechanics and Analysis}, 226(2):497--558,    2017.
	

	\bibitem{BuDu}
	Budhiraja, A.~ and Dupuis, P.~
	\newblock A variational representation for positive functionals of infinite
	dimensional {B}rownian motion.
	\newblock {\em Probab. Math. Statist.}, 20(1):39--61, 2000.
	
	
	\bibitem{AB+PD+VM}
	Budhiraja, A.~, Dupuis, P.~ and Maroulas, V.~
	\newblock Large deviations for infinite dimensional stochastic dynamical
	systems.
	\newblock {\em Ann. Probab.}, 36(4):1390--1420, 2008.
	

	\bibitem{CLT} Cao, H.,Lunasin,  E.~M.~ and Titi, ~E.~S.
	Global well-posedness of three-dimensional viscous and inviscid simplified Bardina turbulence models. {\em Commun. Math. Sci.} 4(4):823--848, 2006.  
	
	
	
	\bibitem{CAO-TITI} 
	Cao, Y. and Titi, E.S. On the rate of convergence of the two-dimensional $\alpha$-models of turbulence to the Navier-Stokes equations. \emph{Numerical Functional Analysis and Optimization}, 30(11-12):1231-1271, (2009).
	
	\bibitem{CARABALLO1} 
	Caraballo, G.~T.~, Real, A.~J. ~, Taniguchi T.~ On the existence and uniqueness of solutions to stochastic three-dimensional Lagrangian averaged Navier-Stokes equations. \emph{Proceedings-Royal Society. Mathematical, Physical and Engineering Sciences}, 462 (2066):459-479, 2006. 

	\bibitem{Chen-1}
	Chen, S., Foias, C., Holm, D.~D., Olson, E., Titi, E.~S. and  Wynne, S.
	Camassa--Holm equations as a closure model for turbulent channel and pipe flow. {\em Phys. Rev. Lett.} 81(24):5338--5341, 1998.
	
	\bibitem{Chen-2}Chen, S., Foias, C., Holm, D.~D.~ Olson, E.~ E., Titi, S.~ and  Wynne, S. 
	A connection between the Camassa--Holm equations and turbulent flows in channels and pipes. {\em Phys. Fluid.} 11(8):2343--2353, 1999.
	
	\bibitem{Chen-3}Chen, S.,  Foias, C., Holm, D.~D.,Olson, E., Titi, E.~S.~and   Wynne, S. The Camassa--Holm equations and turbulence. {\em Phys. D} 133(1--4):49--65, 1999.
	
	\bibitem{CHMZ}S.~Chen, D.~D.~Holm, L.~G.~Margolin and R.~Zhang. 
	Direct numerical simulations of the Navier--Stokes alpha model. {\em Phys. D} 133(1--4):66--83, 1999.
	
	
	\bibitem{Chueshov+Millet}
	Chueshov, I.~ and Millet, A.~
	\newblock Stochastic 2{D} hydrodynamical type systems: well posedness and large
	deviations.
	\newblock {\em Appl. Math. Optim.}, 61(3):379--420, 2010.
	
	\bibitem{PC+CF-BlueBook}
	Constantin, P.~ and Foias, C.~
	\newblock {\em Navier-{S}tokes equations}.
	\newblock 
	University of Chicago Press,
	Chicago, IL, 1988.


	\bibitem{DP+JZ-14}
	Da~Prato, G.~ and Zabczyk, J.~
	\newblock {\em Stochastic equations in infinite dimensions.} Vol. 152. 
	\newblock Cambridge University Press, Cambridge, second edition, 2014.
	
	\bibitem{GABU-SANGO} 
	Deugoue, G.~, Sango, M.~ Weak solutions to stochastic 3D Navier-Stokes-$\alpha$ model of turbulence: $\alpha$-asymptotic behavior. \emph{Journal of Mathematical Analysis and Applications}. 384(1):49-62, 2011.
	
	\bibitem{Gabriel} 
	Deugoue, G. and Sango, M.~ On the Stochastic 3D Navier-Stokes-$\alpha$ Model of Fluids Turbulence. \emph{Abstract and Applied Analysis},  2009 :723236, 2009. \url{https://doi.org/10.1155/2009/723236}.
	
	

	\bibitem{Duan+Millet}
	Duan, J.~ and Millet, A.~
	\newblock Large deviations for the {B}oussinesq equations under random
	influences.
	\newblock {\em Stochastic Process. Appl.}, 119(6):2052--2081, 2009.
	
	\bibitem{Dudley}
	Dudley, R.~M.
	\newblock {\em Real analysis and probability.} Vol. 74.  
	\newblock Cambridge University Press, Cambridge, 2002.
	\newblock Revised reprint of the 1989 original.

	
	\bibitem{FOIAS-HOLM-TITI} 
	Foias, C., Holm, D.D. and Titi, E.~S.~ The three dimensional viscous Camassa-Holm equations, and their relation to the Navier-Stokes equations and turbulence theory. \emph{Journal of Dynamics and Differential Equations}, 14(1): 1-35, 2002.
	
	\bibitem{GH}Geurts,~B. and Holm,~D.~D. Leray and LANS-$\alpha$ modeling of turbulent mixing. {\em Journal of Turbulence} 7(10):1--33, 2006. 
	
	
	\bibitem{Guillin}
	Guillin, A.~
	\newblock Averaging principle of {SDE} with small diffusion: moderate
	deviations.
	\newblock {\em Ann. Probab.}, 31(1):413--443, 2003.

	
	\bibitem{HT}Holm,~D.~D. and Titi,~E.~S.  Computational models of turbulence: the LANS-$\alpha$ model and the role of global analysis. {\em SIAM News} 38(7):1--5, 2005.


	\bibitem{Klebaner}
	Klebaner, F.~C. and Liptser, R.~
	\newblock Moderate deviations for randomly perturbed dynamical systems.
	\newblock {\em Stochastic Process. Appl.}, 80(2):157--176, 1999.
	
	\bibitem{LL-0} Layton,~W. and Lewandowski,~R. A high accuracy Leray-deconvolution model of turbulence and its limiting behavior. {\em  Anal. Appl.} 6(01):23--49, 2008.
	\bibitem{Lee+Leila}
	Lee, C.~Y. and Setayeshgar, L.~
	\newblock The large deviation principle for a stochastic {K}orteweg--de {V}ries
	equation with additive noise.
	\newblock {\em Markov Process. Related Fields}, 21(4):869--886, 2015.
	

\bibitem{Liu+Roeckner} Liu, W.~, R\"ockner,  M.~ and  Zhu, X.-C.~
\newblock Large deviation principles for the stochastic quasi-geostrophic
equations.
\newblock {\em Stochastic Process. Appl.}, 123(8):3299--3327, 2013.
	
	\bibitem{LKT}Lunasin,~E.~M., Kurien,~S. and Titi,~E.~S. Spectral scaling of the Leray-$\alpha$ model for two-dimensional turbulence. {\em J. Phys. A Math. Theor.} 41(34):1--10, 2008.
	
	
	\bibitem{LKTT}Lunasin,~E.~M.,Kurien,~S., Taylor,~M. and Titi,~E.~S. A study of the Navier-Stokes-$\alpha$ model for two-dimensional turbulence. {\em Journal of Turbulence} 8:1--21, 2007.
	
	
	\bibitem{MKSM}Mohseni,~K.,Kosovi\'c,~B., Schkoller,~S. and Marsden,~J.~E.
	Numerical simulations of the Lagrangian averaged Navier-Stokes equations for homogeneous isotropic turbulence. {\em Phys. Fluid.} 15(2):524--544, 2003. 
	
	\bibitem{NS}Nadiga,~B., and Skoller,~S. Enhancement of the inverse-cascade of energy in the two-dimensional Lagrangian-averaged Navier-Stokes equations. 13(5):1528--1531, 2001. 
	
	
	\bibitem{MET-82}
	M\'etivier, M.~
	\newblock {\em Semimartingales: a course on stochastic processes.} Vol. 2. Walter de Gruyte, 2011.

	
	\bibitem{Prevot+Rockner}
	Pr\'ev\^ot, C.~ and R\"ockner, M.~
	\newblock {\em A concise course on stochastic partial differential equations.}
	Vol. 1905.
	\newblock Springer, Berlin, 2007.


	\bibitem{PAUL} 
	Razafimandimby, P.~A.~ Viscosity limit and deviations principles for a grade-two fluid driven by multiplicative noise. \emph{Annali di Matematica Pura ed Applicata}, 197(5):1547--1583, 2017.
	

	\bibitem{TEMAMA}  Temam, R. \emph{Navier-Stokes equations: theory and numerical analysis.} Vol. 343. American Mathematical Soc., 2001.
	
	
	\bibitem{TZhang-NSE-MDP}
	Wang, R.~,  Zhai, J.~ and Zhang, T.~
	\newblock A moderate deviation principle for 2-{D} stochastic {N}avier-{S}tokes
	equations.
	\newblock {\em J. Differential Equations}, 258(10):3363--3390, 2015.
	
	\bibitem{Wang-RDE_MDP}
	Wang, R.~ and Zhang, T.~
	\newblock Moderate deviations for stochastic reaction-diffusion equations with
	multiplicative noise.
	\newblock {\em Potential Anal.}, 42(1):99--113, 2015.
	
	
	\bibitem{YANG-ZHAI} 
	Yang, J., Zhai, J.~ Asymptotics of stochastic 2D hydrodynamical type systems in unbounded domains. \emph{ Infinite Dimensional Analysis}, Quantum Probability and Related Topics.  20(03):1750017, 2017.
	
	
	
	\bibitem{Yang+Xueke}
	Yang, L.~  and Pu, X.~
	\newblock Large deviations for stochastic 3{D} cubic {G}inzburg-{L}andau
	equation with multiplicative noise.
	\newblock {\em Appl. Math. Lett.}, 48:41--46, 2015.
	
	\bibitem{TZhang-1}
	Zhai, J.~ and Zhang, T.~
	\newblock Large deviations for stochastic models of two-dimensional second
	grade fluids.
	\newblock {\em Applied Mathematics {\&} Optimization}, 75(3):471--498, 2017.
	
	\bibitem{Zhai-MDP}
	Zhai, J.~,  Zhang, T.~ and Zheng, W.~
	\newblock Moderate deviations for stochastic models of two-dimensional second
	grade fluids.
	\newblock {\em Stochastics and Dynamics}, 0(0):1850026, 0.
	
	
\end{thebibliography}
\end{document}